\documentclass{amsart}
\usepackage{amsmath,amssymb}
\usepackage{pstricks, pst-plot, pst-node, pst-grad, pst-math}
\usepackage{graphicx,xcolor}
\usepackage{arydshln}
\usepackage{slashed}
 \newtheorem{thm}{Theorem}[section]
 \newtheorem{cor}[thm]{Corollary}
 \newtheorem{lem}[thm]{Lemma}
 \newtheorem{prop}[thm]{Proposition}
 \theoremstyle{definition}
 \newtheorem{defn}[thm]{Definition}
 \theoremstyle{remark}
 \newtheorem{rem}[thm]{Remark}
 \newtheorem{ex}[thm]{Example}
 \numberwithin{equation}{section}
 \usepackage{xcolor}
 \usepackage{mathabx}

\newcommand{\ve}[1]{ \widetilde{#1}}	

\newcommand{\conv}[1]{\widehat #1}	
\newcommand{\convm}[1]{\widetilde#1}	

\newcommand{\Hs}{\mathcal{H}}
\newcommand{\Ls}{\mathcal{L}}
\newcommand{\Bs}{\mathcal{B}}

\newcommand{\N}{\mathbb{N}}
\newcommand{\R}{\mathbb{R}}
\newcommand{\C}{\mathbb{C}}

\newcommand{\doma}{\mathcal{C}}

\newcommand{\dom}{\mathop{\mathcal{D}}}
\newcommand{\Ran}{\mathop{\rm ran}}
\newcommand{\supp}{{\rm supp\ }}

\newcommand\ip[2]{\langle #1, #2 \rangle}
\newcommand\iplr[2]{\left\langle #1, #2 \right\rangle}
\newcommand{\parti}{\slashed{\partial}}

\begin{document}

\title[Enclosure of the Numerical Range and Resolvent Estimates]{Enclosure of the Numerical Range and Resolvent Estimates of Non-selfadjoint Operator Functions}

\author{Axel Torshage} 
\address{Mathematisches Institut, Universit\"at Bern, Sidlerstr.\ 5, 3012 Bern, Switzerland}
\email{axel.torshage@math.unibe.ch}



\date{}
\dedicatory{}

\begin{abstract}
In this paper we discuss the relationship between the numerical range of an extensive class of unbounded operator functions and the joint numerical range of the operator coefficients. Furthermore, we derive methods on how to find estimates of the joint numerical range. Those estimates are used to obtain explicitly computable enclosures of the numerical range of the operator function and resolvent estimates. The enclosure and upper estimate of the norm of the resolvent are optimal given the estimate of the joint numerical range.
\end{abstract}

\maketitle

\section{Introduction}
Unbounded operator functions are important in many branches of physics including elasticity, fluid mechanics, and electromagnetics \cite{MR1911850,book,MR3543766}. The spectrum can be used to understand the action of selfadjoint operators but operators and operator functions in classical physics are frequently non-selfadjoint. To comprehend the stability of these systems under small perturbations we need knowledge of the pseudospectrum or the numerical range \cite{MR1335452,MR2155029,MR2359869}.

The closure of the numerical range is a classic enclosure of the spectrum, \cite{MR971506}, and for points in the resolvent set the distance from the numerical range gives an upper bound on the resolvent \cite{MR1839848}. Hence,  the numerical range is an invaluable tool to study behavior of non-selfadjoint operators and operator functions such as, estimates of the location of the spectrum, and where the resolvent is well-behaved. 

The numerical range of matrix functions is studied in \cite{MR1293915,MR1890990}, where geometric properties are investigated.
However, it is in general not possible to analytically compute the numerical range even for matrices of low dimension. Furthermore, for operator functions the numerical range is not convex or even connected in general. Hence, the existing numerical methods for approximating  the numerical range only work for the finite dimensional cases and even for low dimensional problems, the computations are very time consuming. 

In \cite{KL78} the authors present properties of the numerical range for selfadjoint quadratic operator polynomials $\lambda-A-\lambda^2 B$ using the numerical ranges of the operators $A$ and $B$.
In a related work \cite{ATEN} introduced an explicit enclosure of the numerical range $W(T)$ of operator functions of the form:  
\begin{equation}\label{1:eqAT}
	T(\omega):=A-\omega^2-\frac{1}{c-id\omega-\omega^2}B,
\end{equation}
where $A$ and $B$ are selfadjoint operators, $c\geq0$, and $d\geq0$. The presented enclosure of the numerical range is optimal given the numerical ranges of $A$ and $B$ since for each possible pair $W(A)$ and $W(B)$, there exist operators $A$ and $B$ such that the enclosure coincides with the numerical range of $T$. Although \cite{ATEN} considers a very special case, we will show in this paper that the main results for \eqref{1:eqAT} also hold in a much more general setting. We will study operator functions with an arbitrary number of possibly unbounded operator coefficients. 

In \cite{ATEN} the enclosure is deduced without taking the relationship of $A$ and $B$ into consideration.  In  \cite{DAS,HUA,GJK} the \emph{joint numerical range} is studied, which if applied to the operator coefficients, may improve the enclosure of the numerical range significantly.  However, the complexity of computing the joint numerical range of the operator coefficients is usually directly related to the complexity to compute the numerical range of the operator function. Thus to still be able take advantage of the relationship between the operator coefficients, we will in this paper provide methods to obtain outer bounds of the joint numerical range. We generalize the explicit enclosure of pseudospectrum given in \cite{ATEN}, to our operator function, subject to the outer bound of the joint numerical range, and provide a computable estimate of the norm of the resolvent.

The structure of the paper is as follows:
In Section $2$, we define the numerical range of an operator function, the corresponding joint numerical range of the operator coefficients, and describe how these sets relate to each other. For bounded operator functions the relation is trivial but if the operator function is unbounded, additional challenges arise. These challenges are especially severe if the operator function has several unbounded coefficients.
We address this in Lemma \ref{2lemclos} and Lemma \ref{2lemwta}.

In Section $3$ we introduce methods of finding non-trivial outer bounds of the joint numerical range of the operator function's coefficients. In this section we show how common relations between operators can be utilized to obtain bounds on  the joint numerical range. The main results are Proposition \ref{2transprop}, which presents results for relations involving  Borel functions defined on the operator coefficients, and Corollary \ref{3cordom} which considers cases where one operator coefficient is dominated by the other operator coefficients. 

In Section $4$ we present the  enclosure of the numerical range for operator functions and how this set can be computed from the estimates of the joint numerical range of the coefficients. Results, including how to obtain the boundary of the enclosure, are presented  for the case of two operator coefficients and analyzed in Theorem \ref{maint}. In Theorem \ref{maintmult} those results are generalized to more than two operator coefficients. 

In Section $5$ an enclosure of $\epsilon$-pseudospectrum called $\epsilon$-numerical range is given and we derive an enclosure of this set. The main results are Theorem \ref{3emaint} that shows properties of this enclosure  and Proposition \ref{3tbound} that provides an upper estimate of the resolvent that can be computed explicitly for unbounded operator functions.

Throughout this paper, we use the following notation. Let $\omega_{\Re}$ and $\omega_{\Im}$ denote the real and imaginary parts of $\omega$, respectively.  If $\mathcal{M}$ is a subset of an Euclidean space, then $\partial \mathcal{M}$ denotes the boundary of $\mathcal{M}$ and ${\rm Co}(\mathcal{M})$ denotes the convex hull of $\mathcal{M}$.
\section{Numerical range and joint numerical range}
 Let $\Hs$ denote a Hilbert space with the inner product $\ip{\cdot}{\cdot}$ and $A_j$, $j=1,\hdots,n$ denote selfadjoint operators. For simplicity we use the notation $\ve{A}:=(A_1,\hdots,A_n)$ and assume that $\dom(\ve{A}):=\dom(A_1)\cap\hdots\cap \dom(A_n)$ is dense in $\Hs$.
 
  Let $\widehat{T}:\doma\rightarrow \Ls(\Hs)$ denote the  operator function 
\begin{equation}\label{tmulthat}
	\widehat{T}(\omega):=g(\omega)+\sum_{j=1}^nA_jf^{(j)}(\omega), \quad \dom(\widehat{T}(\omega)):=\dom(\ve{A}).
\end{equation}
where  $f^{(j)},g:\mathcal{C}\rightarrow\C$, $j=1,\hdots,n$ and $\mathcal{C}$ is open and dense in $\C$. 
Assume that $\widehat{T}(\omega)$ is closable for $\omega\in\doma$ and let $T(\omega)$ denote its closure.

 Note that if $A_j$ is non-selfadjoint  and bounded in \eqref{tmulthat}, it is possible to write 
\[
	A_j=\frac{A_j+A_j^*}{2}+i\frac{A_j-A_j^*}{2i}=B_j+iC_j,
\]
where $B_j$ and $C_j$ are selfadjoint. Hence, assuming that the bounded operators in \eqref{tmulthat} are selfadjoint is no restriction. 
 
From the definition of $\dom(T(\omega))$ it follows trivially that
\begin{equation}\label{2domrel}
	\dom(T(\omega))\supset\dom(\ve{A}),\quad \omega\in\doma.
\end{equation}

The numerical range of an operator $A\in\Ls(\Hs)$ is defined as
\[
	W(A):=\{\ip{Au}{u}:u\in\dom(A), \|u\|=1\},
\]
and is thus real for selfadjoint operators. The numerical range of an operator function $T$ is the set
\begin{equation}\label{2wtdef}
W(T):=\{\omega\in \mathcal{C}: \exists u\in \dom(T(\omega))\setminus\{0\},\ \ip{T(\omega)u}{u}=0\}.
\end{equation}

In  \cite{ATEN} an enclosure of the numerical range of  \eqref{1:eqAT} is presented and evaluated. The same idea can be used for the closure of the much more general operator function \eqref{tmulthat}. To be able to obtain a tighter enclosure in the general case we first consider the joint numerical range of $\ve{A}$ defined as 
\begin{equation}\label{2defwab}
W(\ve{A}):=\left\{(\ip{A_1u}{u},\hdots,\ip{A_nu}{u}) :  u\in \dom (\ve{A}),\|u\|=1\right\}\subset  W(A_1) \times\hdots\times W(A_n).
\end{equation}
The joint numerical range is in \cite{DAS} studied for bounded operators and in \cite{HUA} for unbounded operators.

\subsection{Relation between $W(T)$ and $W(\ve{A})$} 

Let  $\ve{\alpha}:=(\alpha_1,\hdots,\alpha_n)\in \R^n$  and define the function
\begin{equation}\label{org2fun}
t_{\ve{\alpha}}(\omega):=g(\omega)+\sum_{j=1}^n{\alpha}_j f^{(j)}(\omega), \quad \omega\in\mathcal{C},
\end{equation}
where $f^{(j)}$ and $g$ are defined as in \eqref{tmulthat}. For a given set $X\subset \R^n$, let $W_X(T)\subset\doma$ denote the set
\begin{equation}\label{1encg}
	W_X(T):=\{\omega\in \mathcal{C}: \exists \ve{\alpha}\in X, t_{\ve{\alpha}}(\omega)=0\}.
\end{equation}
From this definition it follows directly that if $\dom(\ve{A})=\dom(T(\omega))$, $\omega\in\doma$ then
\[
	W(T)=W_{W(\ve{A})}(T).
\]
Hence, the numerical range of $T$ is closely related to the joint numerical range of $\ve{A}$. This was discussed by P. Psarrakos, \cite{PSA}, for operator polynomials in the finite dimensional case. These results can be straight forward generalized to bounded operator functions and operator functions with one unbounded operator  $A_j$ where $f^{(j)}(\omega)\neq0$, $\omega\in\doma$ since in these cases $\dom(\ve{A})=\dom(T(\omega))$. 

  However, for more general unbounded operator functions,  \eqref{2domrel} is often not an equality and then we only have the inclusion $W(T)\supset W_{W(\ve{A})}(T)$. 
  
\begin{rem}\label{imprem}
Even  if $X$ is closed it might not hold that $W_X(T)$ is closed. This can be seen from the example
\[
	T(\omega):=\omega A-I, \quad X:=W(A)=[1,\infty),\quad \dom (T(\omega))=\left\{\begin{array}{c c}
	\dom(A)& \omega\neq0\\
	\Hs& \omega=0
	\end{array}\right ..
\]
Then $W(T)=W_X(T)=(0,1]$, which is not a closed set. 
\end{rem}

For each set $\Omega\supset W(\ve{A})$ it follows trivially from the definition \eqref{1encg} that $W_{\Omega}(T)\supset W_{W(\ve{A})}(T)$. Hence, the goal is to find small $\Omega\supset W(\ve{A})$ such that  $W_\Omega(T)\supset W(T)$. 

\begin{lem}\label{2lemclos}
Assume that $u\in \dom(T(\omega))$, $\|u\|=1$ and $\ip{T(\omega)u}{u}=0$. Then for each $\epsilon>0$ there are $v\in \dom(\ve{A})$, $\|v\|=1$ and $\ve{\alpha}\in \overline{W(A_1)}\times\hdots\times\overline{W(A_n)}$ such that 
\begin{equation}\label{2clos}
	\sum_{j=1}^n|\ip{A_jv}{v}-\alpha_j|<\epsilon,\quad \|u-v\|<\epsilon\quad  \text{and}\quad t_{\ve{\alpha}}(\omega)=0.
\end{equation}
Furthermore, if $u\in\dom(\ve{A})$ then $\omega\in W_{W(A)}(T)$. 
\end{lem}
\begin{proof}
Let, $\{v_i\}_{i=1}^{\infty}$ be a sequence of unit vectors in $\dom(\ve{A})$ such that $\lim_{i\rightarrow \infty}\ip{T(\omega)v_i}{v_i}=0$ and $\lim_{i\rightarrow \infty}\|u-v_i\|=0$. For $j\in\{1,\hdots n\}$ we can assume without loss of generality that if $\lim_{i\rightarrow\infty}\ip{A_jv_i}{v_i}$ does not exist, then $\{\ip{A_jv_i}{v_i}\}_{i=1}^{\infty}$ has no converging subsequence.
Let  $J$ denote the set $J:=\{j\in \{1,\hdots,n\}: \lim_{i\rightarrow\infty} \ip{A_jv_i}{v_i}\text{ does not exist}\}$. For $j\in\{1,\hdots,n\}\setminus J$ define $\alpha_j:=\lim_{i\rightarrow\infty} \ip{A_jv_i}{v_i}\in \overline{W(A_j)}$. Let $i_1$ be a constant such that
\begin{equation}\label{2clos1}
	\sum_{j\in \{1,\hdots,n\}\setminus J} |\ip{A_jv_i}{v_i}-\alpha_j|<\frac{\epsilon}{2},\quad \|u-v_i\|<\epsilon  \quad\text{for} \quad i\geq  i_1,
\end{equation}
and define the constant
\[
	K:=-\sum_{j\in \{1,\hdots,n\}\setminus J}f^{(j)}(\omega)\alpha_j.
\]
Then,
\[
	\ip{T(\omega)u}{u}=\iplr{\overline{\sum_{j\in J_0}f^{(j)}(\omega)A_j}u}{u }-K=0,\quad J_0:=\{j\in J: f^{(j)}(\omega)\neq0\}.
\] 
We then obtain
\[
	K=\lim_{i\rightarrow \infty}\sum_{j\in J_0}f^{(j)}(\omega)\alpha_j^{(i)},\quad \alpha_j^{(i)}:=\ip{A_jv_i}{v_i}.
\]
As $\lim_{i\rightarrow \infty}\ip{A_jv_i}{v_i}$ does not exist for $j\in J_0$ and that $f^{(j)}(\omega)\neq0$ for $j\in J_0$ it follows that $J_0$ either is empty or has at least two elements. 

If $J_0$ is empty define  $v:=v_{i_1}$ and $\alpha_j=\ip{A_jv}{v}$ for $j\in J$. It then follows that $K=0$, and that \eqref{2clos} holds from \eqref{2clos1}.

If there are at least $2$ elements in $J_0$ then since $\C$ can be seen as a two dimensional vectors space over the field $\R$, it follows that for some $k,l\in J_0$ there are constants $\kappa_j,\gamma_j\in \R$, $j\in J_0$ such that $f^{(j)}(\omega)=\kappa_jf^{(k)}(\omega)+\gamma_jf^{(l)}(\omega)$. Here we choose $\kappa_k=\gamma_l=1$ and $\kappa_l=\gamma_k=0$. Thus
\begin{equation}\label{2sumclos}
K=\lim_{i\rightarrow \infty}\left(f^{(k)}(\omega)\sum_{j\in J_0\setminus\{l\}}\kappa_j\alpha_j^{(i)}+f^{(l)}(\omega)\sum_{j\in J_0\setminus\{k\}}\gamma_j\alpha_j^{(i)}\right).
\end{equation}
Since, $K$ is written as the limit of two sums, either both sums converge to a bounded limit or neither of the sums have a limit.  

First consider the case when both sum converge, then  $K=f^{(k)}(\omega)K_k+f^{(l)}(\omega)K_l$ where, $K_k$ and $K_l$ are the real numbers
\begin{equation}\label{2twoconv}
K_k=\lim_{i\rightarrow \infty}\sum_{j\in J_0\setminus\{l\}}\kappa_j\alpha_j^{(i)},\quad  K_l=\lim_{i\rightarrow \infty} \sum_{j\in J_0\setminus\{k\}}\gamma_j\alpha_j^{(i)}.
\end{equation}
Let $i_2\geq i_1$ denote a number such that
\begin{equation}\label{2clos2}
\left |\alpha_{k}^{(i_2)}-\left(K_k-\sum_{j\in J_0\setminus\{k,l\}}\kappa_j\alpha_j^{(i_2)}\right)\right|+\left |\alpha_{l}^{(i_2)}-\left(K_l-\sum_{j\in J_0\setminus\{k,l\}}\gamma_j\alpha_j^{(i_2)}\right)\right|<\frac{\epsilon}{2},
\end{equation}
and
\[
K_k-\sum_{j\in J_0\setminus\{k,l\}}\kappa_j\alpha_j^{(i_2)}\in \overline{W(A_k)}, \quad K_l-\sum_{j\in J_0\setminus\{k,l\}}\gamma_j\alpha_j^{(i_2)}\in \overline{W(A_l)}.
\]
An $i_2$ satisfying the latter properties exists  due to the limit \eqref{2twoconv} and that there are no subsequence of $\{\alpha_{k}^{(i)}\}_{i=1}^{\infty}$ converging to endpoints of $\overline{W(A_k)}$  or  $\{\alpha_{l}^{(i)}\}_{i=1}^\infty$ converging to endpoints of $\overline{W(A_l)}$. Define $\alpha_j:=\alpha_j^{(i_2)}$ for $j\in J\setminus \{k,l\}$ and 
\[
	\alpha_k:=K_k-\sum_{j\in J_0\setminus\{k,l\}}\kappa_j\alpha_j,\quad \alpha_l:=K_l-\sum_{j\in J_0\setminus\{k,l\}}\gamma_j\alpha_j.
\]
It then follows that $\ve{\alpha}\in \overline{W(A_1)}\times\hdots\times\overline{W(A_n)}$, $t_{\ve{\alpha}}(\omega)=0$, \eqref{2clos1} and \eqref{2clos2} yield that \eqref{2clos} holds for $v:=v_{i_2}$. 

Now consider the case when the two sums in \eqref{2sumclos} do not converge. Then
\[
	{\rm Im} \left(\frac{K}{f^{(l)}(\omega)}\right)=\lim_{i\rightarrow \infty}{\rm Im}\left(\frac{f^{(k)}(\omega)}{f^{(l)}(\omega)}\right)\sum_{j\in J_0}\kappa_j\alpha_j^{(i)}.
\]
Since the sum does not converge it follows that $f^{(k)}(\omega)=r f^{(l)}(\omega)$ for some $r\in \R\setminus\{0\}$ and consequently
\[
	\frac{K}{f^{(l)}(\omega)}=\lim_{i\rightarrow \infty} \sum_{j\in J_0}(r\kappa_j+\gamma_j)\alpha_j^{(i)}.
\]

The result is now shown similarly as in the case when the sums converges: let  $i_2\geq i_1$ denote a number such that
\begin{equation}\label{2clos3}
	\left |\alpha_{l}^{(i_2)}-\left(\frac{K}{f^{(l)}(\omega)}-\sum_{j\in J_0\setminus\{l\}}(r\kappa_j+\gamma_j)\alpha_j^{(i_2)}\right)\right|<\frac{\epsilon}{2},
\end{equation}
and
\[
	\alpha_l:= \frac{K}{f^{(l)}(\omega)}-\sum_{j\in J_0\setminus\{l\}}(r\kappa_j+\gamma_j)\alpha_j^{(i_2)}\in \overline{W(A_l)}.
\]
Define $\alpha_j:=\alpha_j^{(i_2)}$ for $j\in J\setminus \{l\}$  and $v:=v_{i_2}$,  \eqref{2clos} is then obtained from \eqref{2clos1} and \eqref{2clos3}. The last statement of the lemma follows direct from definition. 
\end{proof}

\begin{cor}\label{2corclos}
Let $\epsilon>0$ and define the set
\[
	\Omega:=\left\{\ve{\alpha}\in \overline{W(A_1)}\times\hdots\times \overline{W(A_n)}: \exists \ve{\alpha}'\in \overline{W(A_1,\hdots,A_n)}, \sum_{j=1}^n |\alpha_j-\alpha_j'|\leq\epsilon\right\},
\]
then $W_\Omega(T)\supset W(T)$. 
\end{cor}
\begin{proof}
Immediate from Lemma \ref{2lemclos}.
\end{proof}

\begin{ex}
In many cases it is easier to construct an $\Omega\supset W(\ve{A})$ that differs from the one given in Corollary \ref{2corclos}. Assume that $A_j$ in \eqref{tmulthat}  is bounded for $j>1$. Then $\dom (\ve{A})=\dom (A_1)$ and from definition of $T$ it follows that 
\[
	\dom (T(\omega))=\left\{\begin{array}{r r}
	\dom (A_1),& f^{(1)}(\omega)\neq 0\\
	\Hs,& f^{(1)}(\omega)=0\\
	\end{array}\right. .
\]
Hence,  $W(T)=W_{W(\ve{A})}(T)$  on the set $\Gamma:=\{\omega\in\doma :f^{(1)}(\omega)\neq0\}$ and for $\alpha_1\in W(A_1)$ we have that
\[
	W(T)\setminus \Gamma=W_{\{\alpha_1\}\times W(\ve{A}_1)}(T)\setminus\Gamma,\quad \quad \ve{A}_1:=(A_2,\hdots,A_n).
\]
This means that if
\[
	\Omega\supset W(A_1,\hdots,A_n)\cup\left\{(\alpha_1,\ip{A_2u}{u},\hdots,\ip{A_nu}{u}): u\notin \dom (A_1), \|u\|=1\right\},
\]
for any $\alpha_1\in W(A_1)$, then $W(T)\subset W_\Omega(T)$. However, in this example it is best to investigate the values of $\omega$ such that $f^{(1)}(\omega)\neq 0$ and $f^{(1)}(\omega)=0$ separately since then $\dom(T(\omega))$ is constant on the two parts. For values of $\omega$ such that $f^{(1)}(\omega)=0$, we even have that the operator function $T(\omega)$ is  bounded. Additionally, if $f^{(1)}(\omega)=0$, $f^{(1)}$ is holomorphic at $\omega$, and $g$, $f^{(j)}$, $j=2,\hdots,n$ are continuous in  $\omega$ it follows that $\omega\in \overline {W(T)}$. Hence, in many common cases, the values of $\omega$ such that $f^{(1)}(\omega)=0$ are easily investigated. However, if there are more than one unbounded operator this method is not applicable directly. 
\end{ex}

In the case when the functions $g$, $f^{(j)}$, $j=1,\hdots,n$ are holomorphic functions the result of Lemma \ref{2lemclos} can be improved. 

\begin{lem}\label{2lemwta}
Let $T$ denote the closure of the operator function \eqref{tmulthat}, let $t_{\ve{\alpha}}$ be defined by \eqref{org2fun} and denote by $W(T)$ and $W_{W(\ve{A})}(T)$ the sets \eqref{2wtdef} and \eqref{1encg}, respectively. Let $H\subset\doma$ denote the set where $g$, $f^{(j)}$, $j=1,\hdots,n$ are holomorphic and linearly independent. Then $\overline{W(T)}\cap H=\overline{W_{W(\ve{A})}(T)}\cap H$. 
\end{lem}
\begin{proof}
Since $\dom (\ve{A})\subset\dom(T(\omega))$ for all $\omega$, it follows that $\overline{W(T)}\cap H\supset\overline{W_{W(\ve{A})}(T)}\cap H$. Hence, we only have to show the converse. Assume $\omega\in W(T)\cap H$,  and $u\in \dom(T(\omega))$ is a unit vector such that $\ip{T(\omega)u}{u}=0$, then due to Lemma \ref{2lemclos} there is a sequence unit   vectors of  $\{v_i\}_{i=1}^\infty\in \dom(\ve{A})$ and $\ve{\alpha}_i=(\alpha_1^{(i)},\hdots,\alpha_n^{(i)})$ such that

\begin{equation}\label{2closr}
	\sum_{j=1}^n|\ip{A_jv_i}{v_i}-\alpha_j^{(i)}|<\frac{1}{i},\quad \|u-v_i\|<\frac{1}{i}\quad  \text{and}\quad t_{\ve{\alpha}_i}(\omega)=0.
\end{equation}
This implies  that
\[
	t_{\ip{\ve{A}v_i}{v_i}}(\omega)=t_{\ip{\ve{A}v_i}{v_i}}(\omega)- t_{\ve{\alpha}_i}(\omega)=\sum_{j=1}^nf^{(j)}(\omega)\left(\ip{A_jv_i}{v_i}-\alpha_j^{(i)}\right).
\]
If $t_{\ip{\ve{A}v_i}{v_i}}(\omega)=0$, for some $i$ we are done. Otherwise $t_{\ip{\ve{A}v_i}{v_i}}$ is a holomorphic non-constant  function for all $i$ and $t_{\ip{\ve{A}v_i}{v_i}}(\omega)$ is arbitrary small by choosing $i$ large enough. Assume that there exist a smallest  $N\in \N$ such that $|t^{(N)}_{\ip{\ve{A}v_i}{v_i}}(\omega)|\nrightarrow 0$, where $N$ denotes de number of derivatives in $\omega$.
Then since $t_{\ip{\ve{A}v_i}{v_i}}$ is holomorphic it follows that for each $\epsilon>0$ there is an $i$ (large enough) such that for some $\omega'$ satisfying $|\omega'-\omega|<\epsilon$ we have  $t_{\ip{\ve{A}v_i}{v_i}}(\omega')=0$. Hence, the result holds. Now assume that no such $N$ exists. Since the function is holomorphic it follows that $\lim_{i\rightarrow \infty}t_{\ip{\ve{A}v_i}{v_i}}\rightarrow 0$. This leads to a contradiction since $g$ and $f^{(j)}$, $j=1,\hdots, n$ are supposed to be linearly independent. 
\end{proof}

Since $\overline{W_{W(\ve{A})}(T)}=\overline{W(T)}$ in many important cases, such as bounded or holomorphic $T$, we will in the following study enclosures of the sets $W(\ve{A})$ and $\overline{W_{W(\ve{A})}(T)}$ instead of $\overline{W(T)}$. Furthermore, if $T$ is unbounded and non-holomorphic, Lemma \ref{2lemclos} states that it is enough to take an arbitrarily small neighborhood of $W(\ve{A})$, to obtain an enclosure of $W(T)$. 

\subsection{Convex hull of $W(\ve{A})$}
Many studies of the joint numerical range are interested in determining conditions for convexness in $\R^n$,  \cite{DAS, HUA,GJK}. Assume that $\ve{A}=(A_1,A_2)$, where $A_1,A_2$ are selfadjoint and consider $W(\ve{A})$, \eqref{2defwab}. Since $A_1$ and $A_2$ are selfadjoint there is a natural one-to-one correspondence between the numerical range of $A_1+iA_2$ in $\C$ and the joint numerical range of $A_1$ and $A_2$ in $\R^2$. Hence, the joint numerical range of $A_1$ and $A_2$ is convex since the numerical range of an operator is convex.
However, if $\ve{A}=(A_1,\hdots,A_n)$, $n>3$, the joint numerical range is not convex in general, \cite{GJK}. In the case $n=3$ the joint numerical range is convex if $\infty>\dim \Hs>2$, but if $\dim \Hs=2$, this does not necessarily hold. Take for example
\[
	A_1=\begin{bmatrix}
	1& 0\\
	0&-1
	\end{bmatrix},\quad
	A_2=\begin{bmatrix}
	0& 1\\
	1&0
	\end{bmatrix},\quad
	A_3=\begin{bmatrix}
	0& i\\
	-i&0
	\end{bmatrix},
\]
then $W(A_1,A_2,A_3)=\{\ve{\alpha}\in\R^3:\|\ve{\alpha}\|=1\}$, which is the (non-convex) boundary of the unit sphere. 

However, Proposition \ref{2Psats} yields that it is for $W_{W(\ve{A})}(T)$ ultimately irrelevant if $W(\ve{A})$ is convex or not. For operator polynomial $T$ on $\Hs$ of finite dimension, the proposition was shown in \cite[Proposition 2.1]{PSA} and the generalization to our case is straightforward.

\begin{prop}\label{2Psats}
Let $T$ be defined as the closure of \eqref{tmulthat} and let $W(\ve{A})$  and $W_X(T)$ be defined as  \eqref{2defwab} and \eqref{1encg}, respectively. Then
\[	
	W_{W(\ve{A})}(T)=W_{{\rm Co}(W(\ve{A}))}(T),
\]
where ${\rm Co}(W(\ve{A}))$ denotes the convex hull of $W(\ve{A})$.
\end{prop}
\begin{proof}
It is clear that $W_{W(\ve{A})}(T)\subset W_{{\rm Co}(W(\ve{A}))}(T)$. Hence, it is enough to show that  $W_{W(\ve{A})}(T)  \subset  W_{{\rm Co}(W(\ve{A}))}(T)$. Assume that $\omega\in W_{{\rm Co}(W(\ve{A}))}(T)$, then it follows from definition that for some $\ve{\alpha}\in {\rm Co}(W(\ve{A}))$,
\[
	g(\omega)+\sum_{j=1}^n \alpha_jf^{(j)}(\omega)=0.
\]
Since  $\ve{\alpha}\in {\rm Co}(W(\ve{A}))$, Caratheody's Theorem yields that
\[
	\ve{\alpha}=\sum_{i=1}^{n+1}k_i\ve{\alpha}^{(i)},\quad \sum_{i=1}^{n+1}k_i=1,
\] 
for some $\ve{\alpha}^{(i)}\in W(\ve{A})$ and $k_i\geq0$. Since $\ve{\alpha}^{(i)}\in W(\ve{A})$, it follows that
\[
	\ve{\alpha}^{(i)}=\left(\ip{A_1u_i}{u_i},\hdots,\ip{A_nu_i}{u_i}\right),
\]
for some unit $u_i\in \dom (\ve{A})$, $i=1,\hdots,n+1$. Hence, 
\[
	g(\omega)+\sum_{j=1}^n \sum_{i=1}^{n+1}k_i \ip{A_ju_i}{u_i}f^{(j)}(\omega)=0,
\]
and consequently, 
\[
	\sum_{i=1}^{n+1}k_i\left(g(\omega)+\sum_{j=1}^n  \ip{A_ju_i}{u_i}f^{(j)}(\omega)\right)=\sum_{i=1}^{n+1}k_i\ip{T(\omega)u_i}{u_i}=0.
\]
Define the linear operator function $\widetilde{T}(\lambda):=T(\omega)-\lambda$. Then $\ip{T(\omega)u_i}{u_i}\in W(\widetilde{T})$ for $i=1,\hdots,n+1$. Since $\widetilde{T}$ is linear, $W(\widetilde{T})$ is convex and thus $0=\sum_{i=1}^{n+1}k_i\ip{T(\omega)u_i}{u_i}\in W(\widetilde{T})$, which yields that $\omega\in W(T)$, the proposition then follows directly.
\end{proof}

\begin{cor}\label{2Pcor}
Let $W_{W(\ve{A})}(T)$ be defined by \eqref{2defwab} and  \eqref{1encg}. Then
\[
	W_{W(\ve{A})}(T)=\bigcap_{\Omega}W_\Omega(T),
\]
where the intersection is taken over all convex $\Omega\supset W(\ve{A})$.
\end{cor}
\begin{proof}
From Proposition \ref{2Psats} it follows that $W_{W(\ve{A})}(T)=W_{{\rm Co}(W(\ve{A}))}(T)$. The result then is a direct consequence of that ${\rm Co}(W(\ve{A}))$ equals the intersection of all convex  $\Omega\supset W(\ve{A})$.
\end{proof}

\section{Estimates of the joint numerical range}

This section is dedicated to obtain a non-trivial enclosure of $W(\ve{A})$. Computing the set $W(\ve{A})$ is usually as  hard as computing the numerical range of $T$.   The trivial result $\Omega\supset W(\ve{A})\Rightarrow W_\Omega(T)\supset W_{W(\ve{A})}(T)$ implies that if a suitable enclosure $\Omega$ of $W(\ve{A})$ can be found  then an enclosure of $W_{W(\ve{A})}(T)$ is obtained as the solutions $\omega$ of $t_{\ve{\alpha}}(\omega)=0$, $\ve{\alpha}\in\Omega$.

In \cite{ATEN}, the trivial $\Omega:= \overline{W(A_1)}\times\hdots\times \overline{W(A_n)}\subset W(\ve{A})$ is used. This $\Omega$ is easy to investigate and it gives the smallest enclosure of $W(\ve{A})$ without further knowledge of the operators. However, especially when working with more than one unbounded operator coefficient, the enclosure might be rather crude.  The idea is therefore to utilize common relations on the operator coefficients  to be able find a smaller enclosure on $W(\ve{A})$ and thus a better enclosure of $W_{W(\ve{A})}(T)$.
 
\begin{ex}\label{3sadex}
Let $A_1$ and $A_2$ be unbounded selfadjoint operators, where $\dom(A_1)\cap\dom(A_2)$ is dense in $\Hs$. Assume that $W(A_1)=[0,\infty)$ and $W(A_2)=[0,\infty)$ and assume that the unbounded operator polynomial
\[
	 \widehat{T}(\omega):=\omega^2-I-2\omega A_1+2A_2,\quad \dom(\widehat{T}(\omega)):=\dom(A_1)\cap\dom(A_2),
\]
is closable for $\omega\in\C$ and let $T$ denote the closure.
Without further knowledge of $A_1$ and $A_2$, the set $\Omega:=[0,\infty)\times[0,\infty)$ is the smallest enclosure of $W(A_1,A_2)$. It then follows that
\[
	W_\Omega(T)=\left\{\alpha_1\pm\sqrt{\alpha_1^2+1-2\alpha_2}:(\alpha_1,\alpha_2)\in\Omega\right\}=[-1,0]\cup\{\omega\in \C:{\rm Re\  \omega}\geq0\}.
\]
Hence, the enclosure $W_\Omega(T)$ includes the half-plane consisting  of $\omega$ with non-negative real part. However, a better enclosure might be possible to obtain if we have additional knowledge of the correspondence of the operators $A_1$ and $A_2$. Assume that $\ip{A_2u}{u}\leq\ip{A_1u}{u}$ for all $u\in \dom(A_1)\cap\dom(A_2)$ this yields that
\[
	W(A_1,A_2)\subset\Omega':=\{(\alpha_1,\alpha_2)\in [0,\infty)\times[0,\infty): \alpha_2\leq\alpha_1\}.
\]
It then follows that
\[
	W_{\Omega'}(T)=\left\{\alpha_1\pm\sqrt{(\alpha_1-1)^2+2(\alpha_1 -\alpha_2)}:(\alpha_1,\alpha_2)\in\Omega'\right\}=[-1,\infty), 
\]
and since the functions in $T$ are holomorphic, Lemma \ref{2lemwta} yields that $W(T)\subset [-1,\infty)$, which is a significant improvement from $W_\Omega(T)$. 
\end{ex}
Example \ref{3sadex} shows the benefits of using a non-trivial $\Omega$. The condition $\ip{A_1u}{u}\geq\ip{A_2u}{u}$ is useful even if one or both of the operators are bounded. Clearly more complicated  conditions than $\ip{A_1u}{u}\geq\ip{A_2u}{u}$ can be studied.

\begin{lem}\label{4boundlem}
Let $A_j$ be  a selfadjoint operator in $\Hs$, $j=1,\hdots,n$ such that $\dom (\ve{A})$ is dense in $\Hs$. Assume that for some $k\in \{1,\hdots,n\}$ and $M\subset\{1,\hdots,n\}\setminus\{k\}$, there exist functions $y_j:W(A_j)\rightarrow \R $ for $j\in M$, such that one of the following inequalities
 \begin{equation}\label{4bound}
\begin{array}{r l}
	{(\rm i)} &\ip{A_ku}{u}\leq \sum_{j\in M}y_j(\ip{A_ju}{u}),\\
	{(\rm ii)} &\ip{A_ku}{u}\geq \sum_{j\in M}y_j(\ip{A_ju}{u}),
\end{array}
\end{equation}
holds for all $u\in \dom(\ve{A})$.  Then the corresponding  
\begin{equation*}
\begin{array}{r l}
	{(\rm i)} &W(\ve{A})\subset \{\ve{\alpha}\in {W(A_1)}\times\hdots\times {W(A_n)}:\alpha_k\leq \sum_{j\in M}y_j(\alpha_j)\},\\
	{(\rm ii)} &W(\ve{A})\subset \{\ve{\alpha}\in {W(A_1)}\times\hdots\times {W(A_n)}:\alpha_k\geq \sum_{j\in M}y_j(\alpha_j)\},
\end{array}
\end{equation*}
holds.
\end{lem}
\begin{proof}
From \eqref{4bound} the result follows directly. 
\end{proof}

If \eqref{4bound} {(\rm i)} holds and $\{\ve{\alpha}^{(i)}\}_{i=1}^{\infty}\in W(\ve{A})$  it follows that if $\alpha_j^{(i)}$ are bounded as $i\rightarrow \infty$ for $j\in M$, then $\alpha_k^{(i)}$ does not approach $\infty$. In particular if \eqref{4bound} {(\rm i)} holds for $k\in\{1,\hdots, n\}\setminus\{j\}$ with $M=\{j\}$ for some functions $y_j^{(k)}$, $k\in\{1,\hdots,n\}\setminus\{j\}$, then for $\ve{\alpha}\in W(\ve{A})$, $\alpha_k$ is bounded by a constant depending on $\alpha_j$ for $k\in\{1,\hdots,n\}\setminus\{j\}$. In this case, the problem therefore can be studied a bit similar  to the case when there is only one unbounded operator. This stresses the increased importance of relations of the type \eqref{4bound} when there are multiple unbounded operators.

In the following subsections two standard types of relations between the operators in $\ve{A}$ are presented and ways to find  $y_j$ satisfying  \eqref{4bound} in these cases are given. From Lemma \ref{4boundlem} we then obtain non-trivial enclosures of $W(\ve{A})$.

\subsection{Functions of selfadjoint operators}
Let $A\in\Ls(\Hs)$ be a selfadjoint operator. From \cite{MR1501850,MR3122346}, it follows that there exists a sequence of pairwise orthogonal Hilbert spaces $\Hs_i\subset \dom (A)$ with the inner products $\ip{\cdot}{\cdot}_{\Hs_i}=\ip{V_i\cdot}{V_i\cdot}_{\Hs}$, $i=1,2\hdots$, that satisfies
\[
\sum_{i=1}^\infty V_iV_i^* =I_{\Hs},\quad
V_i^*u=\bigg\{\begin{array}{l l}
u,&u\in\Hs_i\\
0,&u\in\Hs_i^\bot
\end{array},
\]
and 
\begin{equation}\label{3sepa}
	A=\sum_{i=1}^\infty V_iA_i V_i^*,\quad  \dom (A)=\left\{u\in\Hs : \sum_{i=1}^\infty \|A_i V_i^*u\|_{\Hs_i}^2<\infty\right\},
\end{equation}
for some selfadjoint operators $A_i\in \Bs(\Hs_i)$.  The sequence of Hilbert spaces $\{\Hs_i\}_{i=1}^{\infty}$ is called a \emph{reduction} of $A$. Since $A_i$ is a bounded selfadjoint operator the spectral theorem \cite[Corollary of Theorem VII.3]{MR0493419}, states that there is some finite measure space $(X_i,\mu_i)$, unitary operator $U_i:L^2(X_i,\mu_i)\rightarrow \Hs_i$  and a multiplication operator $M_{h_i}\in \Bs( L^2(X_i,\mu_i))$  such that
\[
	A_i=U_iM_{h_i}U_i^*.
\]
Here $M_{h_i}$ denotes multiplication with the bounded function $h_i\in L^\infty(X_i,\mu_i)$ and
\begin{equation}\label{3essran}
	\sigma(A_i)=\{x\in\R:\mu(h_i^{-1}(x-\epsilon,x+\epsilon))>0, \text{ for all }\epsilon>0\}.
\end{equation}
The right hand side of \eqref{3essran} is called the essential range of $h_i$. Let $y:\sigma(A)\rightarrow\R$ denote a Borel function bounded on bounded domains. From \eqref{3essran} it follows that the set of $x\in X_i$, where $y(h_i(x))$ is not defined, is of measure $0$. Hence, $M_{y(h_i)}$ is well-defined and we can thus define 
\begin{equation}\label{3yava}
	y(A)=\sum_{i=1}^\infty V_iU_iM_{y(h_i)}U_i^* V_i^*, \quad\dom(y(A))=\left\{u\in\Hs : \sum_{i=1}^\infty \|U_iM_{y(h_i)}U_i^* V_i^*u\|_{\Hs_i}^2<\infty\right\}.
\end{equation}
The operator $y(A)$ is selfadjoint and
\cite[Proposition 3]{MR3122346} implies that $y(A)$ in this definition is independent of the choice of the sequence of Hilbert spaces $\{\Hs_i\}_{i=1}^\infty$.

\begin{prop}\label{3lemjofun}
Let $A$ be a selfadjoint operator and let $y:\sigma(A)\rightarrow\R$ denote a Borel function bounded on bounded domains and define $y(A)$ as in \eqref{3yava}. Then, 
\[
	W(A,y(A))\subset{\rm Co}(\{(\alpha,y(\alpha)):\alpha\in \sigma(A)\}).
\]
\end{prop}
\begin{proof}
Take $(\alpha,\beta)\in W(A,y(A))$ then there is some unit vector $u$ in  $\dom (A)\cap \dom (y(A))$ such that
\[
	(\alpha,\beta)=(\ip{Au}{u},\ip{y(A)u}{u})=\sum_{i=1}^\infty\left(\ip{M_{h_i}z_i}{z_i}_{L^2(X_i,\mu_i)},\ip{M_{y(h_i)}z_i}{z_i}_{L^2(X_i,\mu_i)}\right),
\]
where $z_i:=U_i^*V_i^*u\in L^2(X_i,\mu_i)$. Since the operators $M_{h_i}$ and $M_{y({h_i})}$ are multiplication by $h_i$ respectively $y(h_i)$ in $L^2(X_i,\mu_i)$ this can be written as 
\begin{equation}\label{3genconv}
	(\alpha,\beta)=\sum_{i=1}^\infty\int_{X_i}\left(h_i(x),y(h_i(x))\right)|z_i(x)|^2d\mu_i(x),
\end{equation}
where the integral is defined component wise. Furthermore, since $U_i$ is unitary and $\Hs_i$ are pairwise orthogonal
\[	
	\sum_{i=1}^\infty\int_{X_i}|z_i(x)|^2d\mu_i(x)=\sum_{i=1}^\infty\|z_i\|^2_{L^2(X_i,\mu_i)}=\sum_{i=1}^\infty\|V_i^*u\|_{\Hs_i}^2=\|u\|_{\Hs}^2=1.
\]
Since $h_i(x)\in\sigma(A_i)\subset \sigma(A)$, for all $x\in X_i$ apart from a set of measure $0$, it follows from \eqref{3genconv} and the definition of the convex hull that $(\alpha,\beta)\in {\rm Co}(\{(\alpha,y(\alpha)):\alpha\in \sigma(A)\})$. 
\end{proof}

\begin{cor}\label{3corjofun}
Let $A$ be a selfadjoint operator, with $\overline{W(A)}=[a_0,a_1]$ if $A$ is bounded and let $y:\overline{W(A)}\rightarrow\R$ be a convex Borel function bounded on bounded sets. Furthermore, define $y(A)$ as in \eqref{3yava}. Then
\begin{equation}\label{4conimp}
	W(A,y(A))\subset \{(\alpha,\beta): \alpha \in \overline{W(A)},\ y(\alpha)\leq \beta\leq\conv{y}(\alpha) \},
\end{equation}
where
\[
	\conv{y}(\alpha):=\left\{\begin{array}{c l}\dfrac{y(a_1)-y(a_0)}{a_1-a_0}\alpha+\dfrac{y(a_0)a_1-y(a_1)a_0}{a_1-a_0}& \text{if } A \text{ is bounded}\\
	\infty& \text{if } A \text{ is unbounded}\\
	\end{array}
	\right..
\]
\end{cor}
\begin{proof}
Since $y$ is convex on $\overline{W(A)}$ and for bounded $A$ the function $\conv{y}$ is the straight line between $(a_0,y(a_0))$ and $(a_1,y(a_1))$, it follows that ${\rm Co}(\{(\alpha,y(\alpha)):\alpha\in \overline{W(A)}\})$ is a subset of the proposed enclosure of $W(A,y(A))$. The result then follows directly from Proposition~\ref{3lemjofun} and that ${\rm Co}(\{(\alpha,y(\alpha)):\alpha\in \sigma(A)\})\subset {\rm Co}(\{(\alpha,y(\alpha)):\alpha\in \overline{W(A)}\})$.
\end{proof}

\begin{rem}
Obviously, for concave functions an analog result holds. 
\end{rem}

\begin{ex}\label{3exjopol}
Let $A$ be a  bounded selfadjoint operator, $\overline{W(A)}=[a_0,a_1]$. Assume that either $n\in\N$ is even or $\alpha_0\geq0$. Then  $y(\alpha):=\alpha^n$   is convex and thus from Corollary \ref{3corjofun} we obtain
\begin{equation}\label{4evenimp}
W(A,A^n)\subset\left\{(\alpha,\beta):\alpha\in \overline{W(A)}\ ,\ \alpha^n\leq\beta\leq \alpha\sum_{i=0}^{n-1}a_0^{i}a_1^{n-i-1}-\sum_{i=1}^{n-1}a_0^{i}a_1^{n-i}\right\}.
\end{equation}
 Note that the upper bound on $\beta$ is the straight line between $(\alpha_0,\alpha_0^n)$ and $(\alpha_1,\alpha_1^n)$. 
If $A$ would have been unbounded \eqref{4evenimp} gives only a lower bound on $\beta$. 
In principle this also applies for odd $n$ and indefinite $A$ but then ${\rm Co}(\{(\alpha,y(\alpha)):\alpha\in \overline{W(A)}\})$ is more complicated to describe since $y$ is neither convex nor concave.
\end{ex}

Proposition \ref{3lemjofun} also enables us to find non-trivial enclosures of $\overline{W(\ve{A})}$ in more general cases than $\ve{A}=(A,y(A))$.

\begin{defn}\label{convdef}
Let $y:X\rightarrow\R$ denote a function on a connected set $X\subset\R$. Define the functions $\conv{y}:X\rightarrow \R\cup\{\infty\}$ and $\convm{y}:X\rightarrow \R\cup\{-\infty\}$ as
\[	
\begin{array}{l}
	\conv{y}(x):=\sup \{z:(x,z)\in {\rm Co}(\{x,y(x): x\in \overline{X}\}),\\
	\convm{y}(x):=\inf \{z:(x,z)\in {\rm Co}(\{x,y(x): x\in \overline{X}\}).
	\end{array}
\]
These functions are visualized in Figure \ref{fig:convfunc}.
\begin{center}
	\begin{figure}
	\includegraphics[width=12.5cm]{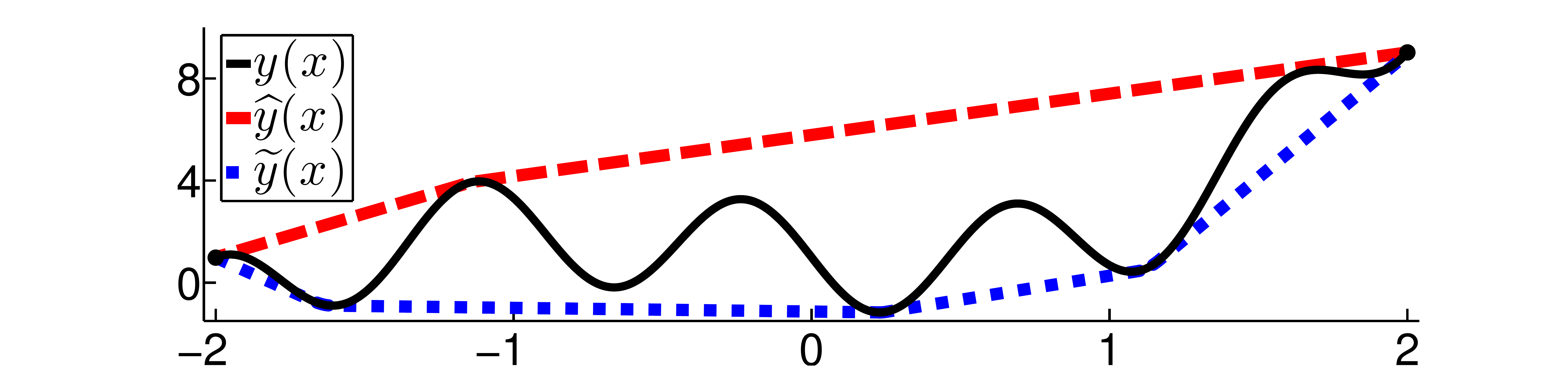}
	\caption{ Visualization of $\conv{y}(x)$ and $\convm{y}(x)$ in Definition \ref{convdef} for $y(x)=x^3+x^2-x+1-2\sin 5x$.}\label{fig:convfunc}
	\end{figure}
\end{center}
\end{defn}

\begin{prop}\label{2transprop}
Let $A_j$ be selfadjoint operators in $\Hs$ for $j=1,\hdots,n$ and assume that $\dom (\ve{A})$ is dense in $\Hs$. Let $k\in \{1,\hdots,n\}$, $M\subset\{1,\hdots,n\}\setminus\{k\}$ and let $y_j:\overline{W(A_j)}\rightarrow \R $, $j\in M$ denote Borel functions bounded on bounded sets. Let $y_j(A_j)$ be defined as in \eqref{3yava} and assume that $\dom(\ve{A})\subset \dom(y_j(A_j))$ for $j\in M$ and  one of the following inequalities holds
 \[ 
\begin{array}{r l}
	{(\rm i)} &\ip{A_ku}{u}\leq \sum_{j\in M}\ip{y_j(A_j)u}{u},\\
	{(\rm ii)} &\ip{A_ku}{u}\geq \sum_{j\in M}\ip{y_j(A_j)u}{u},
\end{array}
\]
for all $u\in \dom(\ve{A})$. Then the corresponding property
\begin{equation}\label{2transeq}
\begin{array}{r l}
	{(\rm i)} &W(\ve{A})\subset \{\ve{\alpha}\in {W(A_1)}\times\hdots\times {W(A_n)}:\alpha_k\leq \sum_{j\in M}\conv{y_j}(\alpha_j)\},\\
	{(\rm ii)} &W(\ve{A})\subset \{\ve{\alpha}\in {W(A_1)}\times\hdots\times {W(A_n)}:\alpha_k\geq \sum_{j\in M}\convm{y_j}(\alpha_j)\},
\end{array}
\end{equation}
holds where $\conv{y_j}$ and $\convm{y_j}$ are given by Definition \ref{convdef} with $X=W(A_j)$.
\end{prop}

\begin{proof}
{(\rm i)} Assume $\ve{\alpha}\in W(\ve{A})$, then there is  some unit $u\in \dom(\ve{A})$ such that $\alpha_j=\ip{A_ju}{u}$ for $j\in M$ and
\[
	 \alpha_k=\ip{A_ku}{u}\leq \sum_{j\in M}\ip{y_j(A_j)u}{u}.
\]
From Proposition \ref{3lemjofun} it follows that for $j\in M$:
\[
	(\alpha_j,\ip{y_j(A_j)u}{u})\in  {\rm Co}(\{(\alpha,y_j(\alpha)):\alpha\in \sigma(A_j)\})\subset{\rm Co}(\{(\alpha,y_j(\alpha)):\alpha\in \overline{W(A_j)}\}),
\] 
which means (by definition) that $\ip{y_j(A_j)u}{u}\leq \conv{y_j}(\alpha_j)$. Then it clearly follows that $\alpha_k\leq \sum_{j\in M}\conv{y_j}(\alpha_j)$, which proves {(\rm i)}. The proof for {(\rm ii)} is analogous.
\end{proof}

Consider the case $\ve{A}=(A_1,A_2)$ and assume that
\[
	\ip{z(A_1)u}{u}\leq\ip{A_2u}{u}\leq \min( \ip{y(A_1)u}{u}, \ip{y'(A_1)u}{u}),\quad u\in\dom(\ve{A})
\]
 for some Borel functions $y,y',z:\overline{W(A_1)}\rightarrow\R$ that are bounded on bounded domains. Then by applying Proposition \ref{2transprop} three times it follows that
\[
	W(\ve{A})\subset \{\ve{\alpha}\in {W(A_1)}\times {W(A_2)}:\convm{z}(\alpha_1)\leq\alpha_2\leq \min (\conv{y}(\alpha_1), \conv{y}'(\alpha_1))\}.
\]
The advantage of being able to use the result in Proposition \ref{2transprop} multiple times is more thoroughly shown in Example \ref{bound:ex}. 

\begin{ex}\label{bound:ex}
Consider the  $\C^{4\times4}$-matrices 
\begin{equation}\label{3exmat}
	A_1:=\begin{bmatrix}
	1\\&2\\& & 2\\& & & 4
	\end{bmatrix},\quad A_2:=\begin{bmatrix}
	& & & 1\\& &1\\& 1\\1
	\end{bmatrix}.
\end{equation}
It is easy to see that $W(A_1,A_2)$ is symmetric with respect to the line $\R\times\{0\}$, thus a non-trivial upper bound also yields a lower bound. By straight forward computations it follows that for unit vectors $u$
\[
	\ip{A_2u}{u}\leq\ip{2^{-s}A_1^su}{u},\quad s\in\R.
\]
Since $\alpha^s$ is a convex function in $\alpha$ for $s\in \R\setminus (0,1)$, it follows from Proposition \ref{2transprop},  (see Corollary \ref{3corjofun}) that
\[
	\ip{A_2u}{u}\leq \frac{2^{s}-2^{-s}}{3}\ip{A_1u}{u}-\frac{2^{s}-2^{-s+2}}{3},\quad s\in \R\setminus (0,1).
\]
In Figure \ref{2fig:jointnum}.(a), the upper (and lower) bound for different $s\in \R\setminus (0,1)$ is visualized.
Since the upper bound holds for $s\in \R\setminus (0,1)$, for each $\ip{A_1u}{u}$ we choose $s$ such that the upper bound is minimized.
This yields the bound $|\ip{A_2u}{u}|\leq z(\ip{A_1u}{u})$ where
\begin{equation}\label{4exbo}
	z(\alpha):=\left\{ \begin{array}{l l}
	\dfrac{2}{3}\sqrt{(4-\alpha)(\alpha-1)}, & \alpha\notin \left(\dfrac{8}{5},\dfrac{5}{2}
\right)\\
	\dfrac{\alpha}{2}&\alpha\in \left(\dfrac{8}{5},2\right]\\
	1& \alpha\in \left(2,\dfrac{5}{2}\right)
\end{array}\right ..
\end{equation}
  In Figure \ref{2fig:jointnum}.(b), the enclosure of $W(A_1,A_2)$ obtained by the bounds on $\ip{A_2u}{u}$ given in \eqref{4exbo} is visualized. It can also be seen that the enclosure actually equals $W(A_1,A_2)$. This shows the power of using multiple conditions when finding the enclosure of $W(A_1,A_2)$. It should here be noted that 
   \[
 	z(A_1)=\begin{bmatrix}
	0\\
	& 1\\
	& & 1\\
	& & & 0
	\end{bmatrix}
 \]
 and thus $\ip{z(A_1)u}{u}<\ip{A_2u}{u}$ for $u=[1,0,0,1]^T$. Hence, the upper bound obtained by taking the intersection of many functions can not be obtained by a single function in general.  

\begin{center}
\begin{figure}
\includegraphics[width=12.5cm]{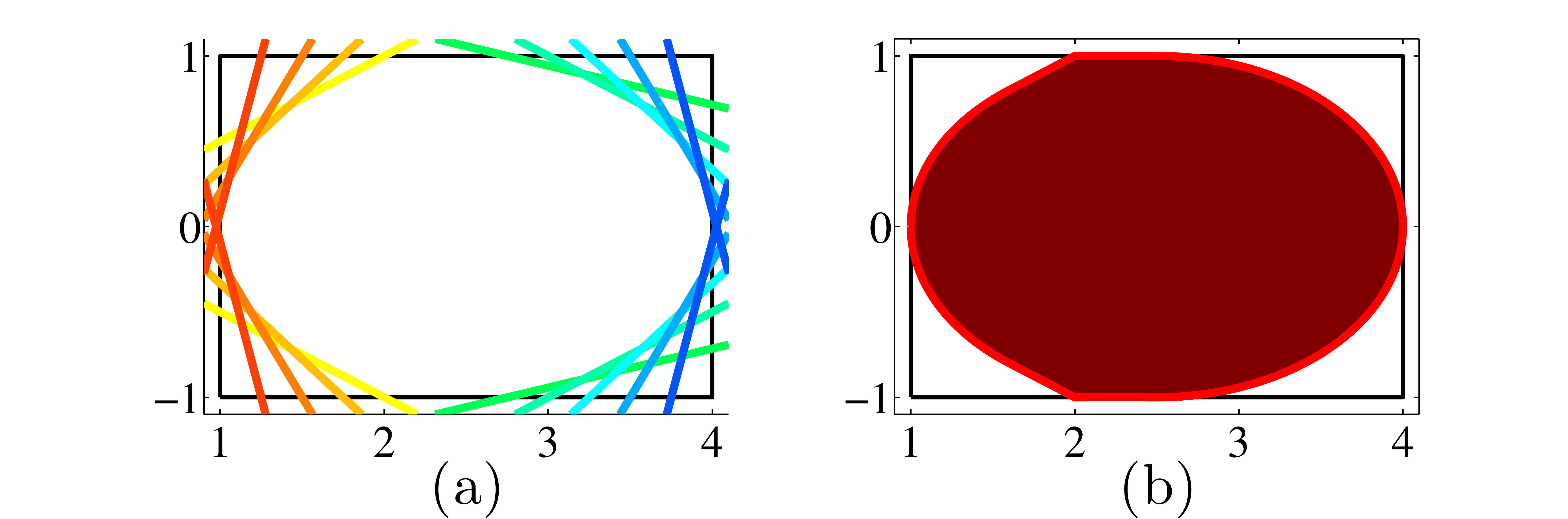}
\caption{The boxes in the figures are $\partial (W(A_1)\times W(A_2))$,  \eqref{3exmat}. In (a) the  lines gives bounds for different $s\in \R\setminus(0,1)$. Panel (b) shows $W(A_1,A_2)$.}\label{2fig:jointnum}
\end{figure}
\end{center}
\end{ex}

\subsection{Domination of selfadjoint operators}
 
For pairs of selfadjoint operators, the case when one of the operators is bounded by the other is commonly studied and a useful relation. The operator $B$ is said to be $A$-bounded if $\dom (A)\subset \dom(B)$ and there are nonnegative constants $\gamma$ and $\gamma'$ such that
 \begin{equation}\label{4unbound}
 	\|Bu\|\leq\gamma\|u\|+ \gamma'\|Au\|,\quad u\in\dom(A).
 \end{equation}
 
It is of interest to see if non-trivial $\Omega\subset W(A)\times W(B)$ can  be obtained from relation \eqref{4unbound}. If $A$ is bounded  then Corollary \ref{3corjofun} can be utilized to find a non-trivial enclosure of $W(A,B)$. From definition \eqref{4unbound} it holds that
 \[
	\ip{B^2u}{u}^{\frac{1}{2}}\leq\gamma\|u\|+\gamma'\ip{A^2u}{u}^{\frac{1}{2}}.
\]
Hence, applying the lower bound of Corollary \ref{3corjofun} for $\ip{B^2u}{u}$ and the upper bound of Corollary \ref{3corjofun} for $\ip{A^2u}{u}$ yields
\[
	\frac{\ip{Bu}{u}}{\|u\|^2}\leq\gamma+ \gamma'\sqrt{\frac{\ip{Au}{u}}{\|u\|^2}(a_0+a_1)-a_0a_1}, 
\]
where $a_0=\inf W(A)$ and $a_1=\sup W(A)$. While this is useful when $A$ is a bounded operator, the bound does not work for unbounded $A$. Hence, for this case we opt for another bound that takes more of the structure of \eqref{4unbound} into consideration to get a non-trivial bound both in the bounded and unbounded case. 

\begin{lem}\label{3lemdom}
Let $A_j$ for $j=1,\hdots,m$ and $B$ be selfadjoint operators in $\Hs$ and assume that $\dom (A_1)\cap\hdots\cap\dom(A_m)$ is dense in $\Hs$. Assume that $A_j\geq0$ and $\ip{A_ju}{A_lu}+\ip{A_lu}{A_ju}\geq 0$, $j,l\in\{1,\hdots m\}$ for $u\in\dom (A_1)\cap\hdots\cap\dom(A_m)$. Furthermore, assume there exists a $\gamma\geq0$, such that for all $u\in\dom (A_1)\cap\hdots\cap\dom(A_m)\subset\dom(B)$ the inequality
\[
	\|Bu\|\leq \gamma\|u\|+ \sum_{j=1}^{m}\|A_ju\|
\]
holds. Then 
\begin{equation}\label{3normbound}
	|\ip{Bu}{u}|\leq \left(\gamma^{\frac{2}{3}} \|u\|^\frac{4}{3}+\sum_{j=1}^{m}  \ip{A_j u}{u}^{\frac{2}{3}}\right)^{\frac{3}{2}}, \quad u\in  \bigcap_{j=1}^m\dom (A_j).
\end{equation}
\end{lem}
\begin{proof}
Define $A_{0}:=\gamma I$ to simplify the notation. We will only show the upper bound, but the proof of the lower limit is completely analogous. First consider the case when $\gamma>0$ and define the operator
\begin{equation}\label{3Cbound}
	C:=\sum_{j=0}^m A_j \sqrt{1+\sum_{l=j+1}^{m}{c_{j,l}}+\sum_{l=0}^{j-1}\frac{1}{c_{l,j}}}, \quad \dom (C):=  \bigcap_{j=1}^m\dom (A_j),
\end{equation}
where, $c_{j,l}>0$ for $0\leq j<l\leq m$  are  some constants.  Since $C$ is a positive sum of positive selfadjoint operators,  $C$ is a selfadjoint operator with the given domain. From definition it follows that for $u\in \dom (C)$
\begin{multline}\label{3badest}
	\|Cu\|\geq\sqrt{\sum_{j=0}^m \|A_ju\|^2 \left(1+\sum_{l=j+1}^{m}{c_{j,l}}+\sum_{l=0}^{j-1}\frac{1}{c_{l,j}}\right) }=\\
	\sqrt{\left(\sum_{j=0}^m   \|A_ju\| \right)^2+\sum_{j=0}^{m-1}\sum_{l=j+1}^{m}\left(\sqrt{c_{j,l}}\|A_ju\|-\frac{1}{\sqrt{c_{j,l}}}\|A_lu\| \right)^2}\geq  \sum_{j=0}^m   \|A_ju\|\geq\|Bu\|.
\end{multline}

Assume that there is an $u\in\dom(C)$ such that $\ip{Bu}{u}>\ip{Cu}{u}$. It then follows that there exists a $k>1$, such that $\ip{Bu}{u}>\ip{kCu}{u}$.
 Define the selfadjoint operator $D:=\overline{B-kC}$, $\dom(D):=\{u\in\Hs: \|Du\|<\infty\}\supset \dom(C)$ and thus
\[
	\ip{Du}{u}=\ip{Bu}{u}-\ip{kCu}{u}>0.
\]

This implies that there is some $\lambda>0$ such that $\lambda\in\sigma(D)$. Let $\{\Hs_i\}_{i=1}^{\infty}$ be a reduction of $D$ as in \eqref{3sepa} and  for $i=1,2,\hdots$ let  $D_i$ denote the operator in $\Hs_i$. Then there is some $i$ such that  $\lambda\in \sigma(D_i)$. Since $D_i$ is a bounded selfadjoint operator it follows by the spectral theorem \cite[Corollary of Theorem VII.3]{MR0493419}, that there exists a  finite measure space $(X,\mu)$, a unitary operator $U:L^2(X,\mu)\rightarrow \Hs_i$  and a multiplication operator $M_{h}\in \Bs( L^2(X,\mu))$ such that
\[
	D_i=UM_hU^*.
\]
Here $M_{h}$ denotes multiplication with the function $h\in L^\infty(X,\mu)$ with essential range $\sigma(D_i)$, see \eqref{3essran}. Since $\lambda\in\sigma(D_i)$ and $\lambda>0$ there is some function $\varphi\in L^2(X,\mu)$ such that 
\[
	\mu(\supp \varphi)>0, \quad\supp \varphi\subset \{x\in X:h(x)>0\}.
\]
This implies that
\[
	M_h\varphi=M_{h_0}\varphi,\quad h_0(x):=\max(h(x),0), 
\]
thus $M_{h_0}$  is a positive semi-definite operator and $\|M_{h_0}\varphi\|_{\Ls^2(X,\mu)}>0$. Define $v:=V_iU\varphi$, then 
\[
	Dv=V_iUM_{h}U^*V_i^*v=V_iUM_{h}\varphi=V_iUM_{h_0}\varphi=Ev, \quad E:=V_iUM_{h_0}U^*V_i^*.
\]
If $v\in\dom(C)$ then this means that $Bv=(kC+E)v$ and thus since $\|Ev\|>0$  and $E$ is positive semi-definite 
\[
	 \|Bv\|=\|(kC+E)v\|>\|Cv\|,
\]
which contradicts \eqref{3badest}.
If $v\notin \dom (C)$ then since $\dom (C)$ is dense in $\dom (D)$ and $v\in V_i\Hs_i$ there is a unit vector  $v'\in\dom (C)\cap V_i\Hs_i$ such that $(k-1)\|Cv'\|>\|D_i\|_{\Hs_i}$. 
Then it follows that
\[
	\|Bv'\|=\|kCv'+Dv'\|\geq k\|Cv'\|-\|D_i\|_{\Hs_i}>\|C v'\|.
\]
Hence, we have a contradiction of \eqref{3badest}. 
Thus, for $u\in\dom (C)$ it holds that
\[
	\ip{Bu}{u}\leq \ip{Cu}{u}=\sum_{j=0}^m \ip{A_j u}{u} \sqrt{1+\sum_{l=j+1}^{m}{c_{j,l}}+\sum_{l=0}^{j-1}\frac{1}{c_{l,j}}}.
\]
This holds for all choices of constants $c_{j,l}>0$. Hence, for each $u\in\dom(A)$ we are interested in the $c_{j,l}>0$ that minimizes $\ip{Cu}{u}$, \eqref{3Cbound} and thus obtain the sharpest possible bound using similar reasoning as in Example \ref{bound:ex}. By simple analysis it follows that the minimizing $c_{j,l}$ are
\[
	c_{j,l}=\left(\frac{ \ip{A_lu}{u}}{ \ip{A_ju}{u}}\right)^{\frac{2}{3}}.
\]
Using these $c_{j,l}$ the result \eqref{3normbound} follows for $\gamma>0$. Let $\gamma=0$ and assume that the result does not hold, then for some $u\in\dom(A_1)\cap\hdots\cap\dom(A_m)$
\[
\left(\sum_{j=1}^{m}\ip{A_j u}{u}^{\frac{2}{3}}\right)^{\frac{3}{2}}<\ip{Bu}{u}\leq\left(\epsilon^{\frac{2}{3}} \|u\|^\frac{4}{3}+\sum_{j=1}^{m} \ip{A_j u}{u}^{\frac{2}{3}}\right)^{\frac{3}{2}},
\]	
for all $\epsilon>0$. But, this is a contradiction since the upper bound on $\ip{Bu}{u}$  is continuous in $\epsilon$, which can be arbitrarily small. 
\end{proof}

If $\gamma=0$ and $m=1$ in Lemma \ref{3lemdom}  it follows that $|\ip{Bu}{u}|\leq\ip{A_1u}{u}$.  Additionally, if $A_1$ is unbounded and $A_j$ are bounded for $j>1$, then \eqref{3normbound} behaves like $\ip{A_1u}{u}$  for large $\ip{A_1u}{u}$. 
\begin{cor}\label{3cordom}
Let $A_j$, $j=1,\hdots,n$ be selfadjoint operators in $\Hs$ and assume that $\dom (\ve{A})$ is dense in $\Hs$. Let $k\in \{1,\hdots,n\}$, $M\subset\{1,\hdots,n\}\setminus\{k\}$, and let 
 $y_j:\overline{W(A_j)}\rightarrow \R $ be Borel functions that are bounded on bounded sets where $y_j(A_j)$ is defined as in \eqref{3yava}.  Assume that $y_j(A_j)\geq0$, $\ip{y_j(A_j)u}{y_l(A_l)u}+\ip{y_l(A_l)u}{y_j(A_j)u}\geq 0$, $j,l\in M$ for $u\in \dom(\ve{A})$ and that there exists an $\gamma\geq0$, such that 
\[
	\|A_ku\|\leq \gamma\|u\| +\sum_{j\in M}\|y_j(A_j)u\|, \quad u\in \dom(\ve{A})\subset\bigcap_{j\in M}\dom(y_j(A_j))\subset \dom(A_k).
\]
Then
\begin{equation}\label{2transbo}
	W(\ve{A})\subset \left\{\ve{\alpha}\in W(A_1)\times\hdots\times W(A_n):|\alpha_k|\leq \left(\gamma^{\frac{2}{3}}+\sum_{j\in M} \conv{y_j}(\alpha_j)^{\frac{2}{3}}\right)^{\frac{3}{2}}\right\},
\end{equation}
where $\conv{y_j}$ is defined as in Definition \ref{convdef} for $W(A_j)$.
\end{cor}
\begin{proof}
From Lemma \ref{3lemdom} it follows that for unit vectors $u\in \dom(\ve{A})$
\[
	|\ip{A_ku}{u}|\leq \left(\gamma^{\frac{2}{3}} \|u\|^\frac{4}{3}+\sum_{j\in M} \ip{y_j(A_j) u}{u}^{\frac{2}{3}}\right)^{\frac{3}{2}}, 
\]
and the result then is given by the inequality  $0\leq \ip{y_j(A_j) u}{u}\leq \conv{y_j}(\ip{A_j u}{u})$ given by  Proposition~\ref{3lemjofun}.
\end{proof}

Corollary \ref{3cordom} and/or Proposition \ref{2transprop} can be used any number of times to obtain a smaller enclosure of $W(\ve{A})$  as in Example \ref{bound:ex}.

\begin{rem}
The converse statement of Corollary \ref{3cordom} does not hold. For example, let $A_1$ and $A_2$ be defined as in Example \ref{bound:ex}, then $\ip{A_1u}{u}\geq 2\ip{A_2u}{u}$ for all $u$ but $\|A_2u\|=\|u\|$ which is larger than $2^{-1}\|A_1u\|$ for some $u\in \C^4$.
\end{rem}

\section{Enclosure of the numerical range}\label{4sect}
This section generalizes  results in \cite[Section 2]{ATEN}, to the closure of the operator function \eqref{tmulthat}, and closed $\Omega$ such that $W(\ve{A})\subset \Omega\subset \overline{W(A_1)}\times\hdots\times \overline{W(A_n)}$. Define the set $W_\Omega(T)$ by \eqref{1encg}:
\begin{equation}\label{2enmu}
W_\Omega(T)=\{\omega\in \mathcal{C}: \exists \ve{\alpha}\in \Omega, t_{\ve{\alpha}}(\omega)=0\}.
\end{equation}
$W_\Omega(T)$ is an enclosure of $W_{W(\ve{A})}(T)$ and from Corollary \ref{2corclos} and Lemma \ref{2lemwta} it follows that $W_\Omega(T)$ is related to $W(T)$ and $\overline{W_\Omega(T)}\supset\overline{W(T)}$ if $T$ is either holomorphic or bounded.

We define the set $\parti\Omega$ iteratively. Let $J_\Omega$ denote the set
\begin{equation}\label{4infcond}
	J_\Omega:=\{j\in\{1,\hdots,n\}:\inf_{\ve{\alpha}\in\Omega} \alpha_j=-\infty \ \&\  \sup_{\ve{\alpha}\in\Omega} \alpha_j=\infty \}
\end{equation}
and define 
\begin{equation}\label{boundinf}
\begin{array}{l c c}
	\parti\Omega=\parti\Omega^+\cup\parti\Omega^-&J_{\Omega}\neq\emptyset &\text{where } \Omega^{\pm}:=\{\alpha\in \Omega:\pm\alpha_{\min J}\geq0\}\\
	\parti\Omega=\partial\Omega& J_{\Omega}=\emptyset
\end{array}.
\end{equation}

Hence,   $\parti\Omega\supset\partial\Omega$, with equality if each operator in $\ve{A}$ is bounded either from below or above.

\begin{lem}\label{genbound}
Let $A_j$, $j=1,\hdots,n$ be selfadjoint operators and let $\Omega$ denote a closed set satisfying  $W(\ve{A})\subset\Omega\subset \overline{W(A_1)}\times\dots\times\overline{W(A_n)}$. Denote by $L$ a straight line in $\R^n$ where $\Omega\cap L\neq\emptyset$. Then $L\cap \parti\Omega\neq \emptyset$, where $\parti\Omega$ is defined in \eqref{boundinf}.
\end{lem}
\begin{proof}
Assume that $L\cap \partial\Omega= \emptyset$, then  $L\subset\Omega$. Consequently $J_\Omega\neq\emptyset$ and thus $L\cap \Omega^{+}\neq\emptyset$ or $L\cap \Omega^{-}\neq\emptyset$, where $\Omega^{\pm}$ is defined in \eqref{boundinf}. If $L\cap\partial \Omega^{\pm}\neq\emptyset$ we are done, otherwise $L\subset\Omega^{\pm}$. It then follows that $J_{\Omega^{\pm}}\neq\emptyset$ and \eqref{boundinf} can be applied again. By doing this iteratively at most $n$ times we obtain a set $\Omega'$ such that $\parti\Omega'\subset\parti\Omega$, $L\cap\Omega'\neq\emptyset$ and $J_{\Omega'}=\emptyset$. It then follows that $L\nsubset\Omega'$ and thus $L\cap\partial\Omega'\subset L\cap\parti\Omega\neq\emptyset$.
\end{proof}

In the following we consider two cases separately.
In the first part we will consider the case with two  operator coefficients as in  \cite{ATEN}. In the second part we consider the case when there are more then two operator coefficients. These cases are studied separately since the results  can be improved  when  $n=2$.

\subsection{Functions with two operator coefficients}\label{sec:gent}
Consider $T$ given as the closure of \eqref{tmulthat} with $n=2$, (in this case $W(\ve{A})$ is convex even though we do not utilize it here).

\begin{prop}\label{det0c} 
Let $T$ be defined as the closure of \eqref{tmulthat} with $n=2$ and  let $\Omega$ be a closed set satisfying $W(A_1,A_2)\subset \Omega\subset \overline{W(A_1)}\times \overline{W(A_2)}$.
Assume that ${\rm Im}(f^{(1)}(\omega)\overline{f^{(2)}(\omega)})=0$ and let $W_\Omega(T)$ denote the set \eqref{2enmu}. 
Then, $\omega\in W_\Omega(T)$ if and only if the degenerate system 
\begin{equation}\label{21ran}
\begin{bmatrix}
f^{(1)}_\Re(\omega) & f^{(2)}_\Re(\omega)\\
f^{(1)}_\Im(\omega) & f^{(2)}_\Im(\omega)
\end{bmatrix}
\begin{bmatrix}
\alpha_1\\
\alpha_2
\end{bmatrix}=-
\begin{bmatrix}
g_\Re(\omega)\\
g_\Im(\omega) 
\end{bmatrix},
\end{equation}
has a solution in $\parti\Omega$.

\end{prop}
\begin{proof}
Assume that $\omega\in W_{\Omega}(T)$, then $t_{(\alpha_1,\alpha_2)}(\omega)=0$ for some $(\alpha_1,\alpha_2)\in \Omega$, which thus is a solution to the system \eqref{21ran}. Since, ${\rm Im}(f^{(1)}(\omega)\overline{f^{(2)}(\omega)})=0$ this system is degenerate. Hence, there is a line of solutions to this problem in $\R^2$. Lemma \ref{genbound} yields that \eqref{21ran} has a solution in $\parti\Omega$. The converse statement follows directly.
\end{proof}

\begin{prop}\label{detc}
Let  $T$ be defined as the closure of \eqref{tmulthat} with $n=2$ and let  $\Omega$ be a closed set satisfying $W(A_1,A_2)\subset \Omega\subset \overline{W(A_1)}\times \overline{W(A_2)}$. 
Assume that ${\rm Im}(f^{(1)}(\omega)\overline {f^{(2)}(\omega)})\neq0$, and let $W_\Omega(T)$   denote the set  \eqref{2enmu}. Then $\omega\in W_\Omega(T)$ if and only if
\begin{equation}\label{result}
\left(\frac{{\rm Im}(f^{(2)}(\omega)\overline{g(\omega)})}{{\rm Im}(f^{(1)}(\omega)\overline{f^{(2)}(\omega)})},\frac{{\rm Im}(\overline{f^{(1)}(\omega)}g(\omega))}{{\rm Im}(f^{(1)}(\omega)\overline{f^{(2)}(\omega)})}\right)\in\Omega.
\end{equation}
\end{prop}
\begin{proof}
Since $\alpha_1$, $\alpha_2$ are real, the real and imaginary part of $t_{(\alpha_1,\alpha_2)}(\omega)$ gives a 2-dimensional linear problem in $\alpha_1$ and $\alpha_2$. Solving it yields \eqref{result}.
\end{proof}
Define  for $\mathcal{C}$ the partition
\begin{equation}\label{2parti}
 \mathcal{C}^i=\{\omega\in \mathcal{C}: {\rm Im}(f^{(1)}(\omega)\overline {f^{(2)}(\omega)})=0\},\quad
 \mathcal{C}^r= \mathcal{C} \setminus \mathcal{C}^i.
\end{equation}

\begin{thm}\label{maint} 
Let $T$ be defined as the closure of \eqref{tmulthat} with $n=2$ and assume that $\Omega$ is a closed set satisfying $W(A_1,A_2)\subset \Omega\subset \overline{W(A_1)}\times \overline{W(A_2)}$. Denote by 
$W_\Omega(T)$, $W_{\parti\Omega}(T)$ and $W_{\partial\Omega}(T)$, sets given by \eqref{1encg}. Further, denote by $C$ the set where $f^{(1)},f^{(2)},g$ are continuous and  let $H$ be the set where those functions are holomorphic and linearly independent. Then
\leavevmode
\begin{itemize}
\item[{\rm (i)}] $W_{\parti\Omega}(T)\cap \mathcal{C}^i=W_{\Omega}(T)\cap \mathcal{C}^i$,
\item[{\rm (ii)}] $ W_{\partial\Omega}(T)\cap \mathcal{C}^r\cap C\supset \partial W_{\Omega}(T)\cap \mathcal{C}^r\cap C$,
\item[{\rm (iii)}] $ W_{\partial\Omega}(T)\cap \mathcal{C}^r\cap H = \partial W_{\Omega}(T)\cap \mathcal{C}^r\cap H$.
\end{itemize}
\end{thm}
\begin{proof}
{\rm (i)}
Assume $\omega\in W_\Omega(T)\cap C^i$, then from Proposition \ref{det0c} it follows that $\omega\in W_{\parti \Omega}(T)\cap C^i$, and the converse is trivial.

 {\rm (ii)} 
Let $\omega\in\partial W_{\Omega}(T)\cap\mathcal{C}^r\cap C$,  due to Proposition \ref{detc} there is a unique pair $(\alpha_1,\alpha_2)\in \R^2$ such that $f^{(1)}(\omega)\alpha_1+f^{(2)}(\omega)\alpha_2+g(\omega)=0$. Then due to continuity of the functions $f^{(1)},f^{(2)},g$ and from continuity of the solution of a non-degenerate matrix equation it follows that for each $\omega'$ in some small open ball around $\omega$ there is a (by Proposition \ref{detc}) unique solution $(\alpha_1',\alpha_2')$ to $f^{(1)}(\omega')\alpha_1'+f^{(2)}(\omega')\alpha_2'+g(\omega')=0$, that is close to $(\alpha_1,\alpha_2)$. But since $\omega\in\partial W_{\Omega}(T)$ this implies that  for each open ball at least one of the solutions is in $\Omega$ and at least one in not in $\Omega$. Hence, for some $\omega'$ close to $\omega$  the unique solution is some $(\alpha_1',\alpha_2')\in\partial\Omega$, and since this holds for all balls around $\omega$, it follows that $\omega\in W_{\partial\Omega}(T)\cap \mathcal{C}^r\cap C$. 

{\rm (iii)}
Assume $\omega\in  W_{\partial\Omega}(T)\cap \mathcal{C}^r\cap H$, then there are some $(\alpha_1,\alpha_2)\in\partial\Omega$  such that $t_{(\alpha_1,\alpha_2)}(\omega)=f^{(1)}(\omega)\alpha_1+f^{(2)}(\omega)\alpha_2+g(\omega)=0$. The zeros of a holomorphic function are continuous in holomorphic perturbations. Hence, it follows that for each $r>0$, there is some $r'>0$ such that $\sqrt{(\alpha_1-\alpha_1')^2+(\alpha_2-\alpha_2')^2}<r'$ implies that $t_{(\alpha_1',\alpha_2')}$ has a zero in $B(\omega,r)$. Where  $B(\omega,r)$ is the open ball with centrum $\omega$ and radius $r$. Since $\mathcal{C}^r$ is open,  for $r>0$ small enough $B(\omega,r)\subset\doma^{r}$. This means that $(\alpha_1,\alpha_2)\in\partial\Omega$ yields $\omega\in \partial W_{\Omega}(T)$. Hence, $ W_{\partial\Omega}(T)\cap \mathcal{C}^r\cap H \subset \partial W_{\Omega}(T)\cap \mathcal{C}^r\cap H$, the converse statement  follows from  {\rm (ii)}.
\end{proof}

If $f^{(1)},f^{(2)},g$ are holomorphic in $\mathcal{C}$ then Theorem \ref{maint} {\rm (ii)} is a consequence of Theorem \ref{maint} {\rm (iii)}. However, 
equality does not hold in general in Theorem \ref{maint} {\rm (ii)}.
\begin{ex} 
 Let $\Omega= W(A_1)\times W(A_2)=[0,1]^2$ and define the operator function
\[
	T(\omega):=A_1+A_2(\omega+2i)+g(\omega),\quad 
	g(\omega):=\Bigg\{
	\begin{array}{l l}
	0 &\text{for } |\omega|<1 \\
	\omega^2(|\omega|-1)&\text{for } 1\leq|\omega|<2 \\
	\omega^2& \text{for } |\omega|\geq2
	\end{array}.
\]
The functions $f^{(1)},f^{(2)},g$ are continuous and $(0,0)\in \partial\Omega\cap C^r$ but it follows from definition that $W_{\{(0,0)\}}(T)=\overline{B(0,1)}$, where $\overline{B(0,1)}$ is the closed unit disc. Hence, $W_{\partial\Omega}(T)\cap \mathcal{C}^r\cap C\neq \partial W_{\Omega}(T)\cap \mathcal{C}^r\cap C$ in this case. 
\end{ex}

\begin{rem}
An algorithm similar to  \cite[Proposition 2.17]{ATEN} can be used to obtain  $\overline{ W_{\Omega}(T)\cap \mathcal{C}^r}$ from its boundary. However, the algorithm does not necessary converge in a finite number of steps since there can be an infinite number of components. Figure \ref{2fig:sinclust} shows  this behavior in a simple case.
\end{rem}

\subsection{Functions with more than two operator coefficients} 
Consider $T$ given as the closure of \eqref{tmulthat} with $n>2$. A major difference in this case is that $t_{\ve{\alpha}}(\omega)=0$ does not have a unique solution $\ve{\alpha}\in\R^n$ for any $\omega\in\mathcal{C}$. Hence,  whether $\omega\in W_\Omega(T)$ or $\omega\notin W_\Omega(T)$ is not determined by checking if a unique point $\ve{\alpha}\in\mathcal{C}$ is in $\Omega$. The solvability and the solutions  of  $t_{\ve{\alpha}}(\omega)=0$ trivially coincides with the solvability and the solutions of
\[
	\begin{bmatrix}
	f^{(1)}_\Re(\omega) & \dots &f^{(n)}_\Re(\omega)  \\
	f^{(1)}_\Im(\omega)  & \dots &f^{(n)}_\Im(\omega)  
	\end{bmatrix}
	\begin{bmatrix}
	\alpha_1 \\
	\vdots\\
	\alpha_n
	\end{bmatrix}=
	-
	\begin{bmatrix}
	g_\Re(\omega)  \\
	g_\Im(\omega) 
	\end{bmatrix}.
\]
We will denote this system by the  matrix notation:
\begin{equation}\label{multeq}
	F(\omega)\ve{\alpha}=G(\omega).
\end{equation}

\begin{defn}\label{mskel}
Let $\Omega\subset \R^n$ be a closed possibly infinite set. The $m$-skeleton, $\Omega_m$ for $m=0,\hdots,n-1$ is then defined iteratively as follows:

Let $J_{\Omega}\subset\{1,\hdots,n\}$ be defined as in \eqref{4infcond}.

If $J_{\Omega}\neq\emptyset$ define
\[
	\Omega_m:=\Omega^{+}_m\cup\Omega^{-}_m,\quad \quad \Omega^{\pm}:=\{\alpha\in \Omega:\pm\alpha_{\min J}\geq0\}.
\]

If $J_{\Omega}=\emptyset$ define $\Omega_0=\partial \Omega$ and for $m>0$ let $\Omega_m$ denote the set of points $\ve{\alpha}\in \Omega$ such that there is a $m$-dimensional subspace $P$ satisfying $\Omega\cap P=\{\ve{\alpha}\}$. For convenience we define $\Omega_{-1}:=\Omega$. 
\end{defn}

Definition \ref{mskel} is a generalization of \eqref{boundinf}, note that $\parti\Omega=\Omega_0$. Furthermore, it follows that 
\[
	\Omega=\Omega_{-1}\subset\Omega_0\subset\Omega_1\subset\hdots\subset\Omega_{n-1}.
\]

\begin{ex}\label{4skelex}
If $\Omega=\overline{B(0,r)}$, the open ball around $0$ with radius $r$, then $\Omega_m=\partial B(0,r)$ for $m=0,\dots,n-1$. If $\Omega=[0,r]^n$, then 
\[
\Omega_m:=\{\ve{\alpha}\in\Omega: \sum_{\{j\in\{1,\hdots, n\}:\alpha_j\in\{0,r\}\}}1\geq m+1\},\quad  m=0,\hdots,n-1.
\]
\end{ex}

\begin{lem}\label{lemfin1}
Let $\Omega\subset \R^n$ be a closed set and $P$ a $m$-dimensional subspace, where $m\in \{0,\hdots,n-1\}$. If $\Omega\cap P\neq \emptyset$ then $\Omega_{m-1}\cap P\neq\emptyset$.
\end{lem}
\begin{proof}
For $m=0$ this is direct from definition and for $m=1$ it follows from Lemma \ref{genbound}. Hence, for the rest of the proof, assume that $m\geq2$.

Let $J_{\Omega}\subset\{1,\hdots,n\}$ be defined as in \eqref{4infcond} and if $J_\Omega\neq\emptyset$ define
\[
	\Omega:=\Omega^{+}\cup\Omega^{-},\quad \quad \Omega^{\pm}:=\{\alpha\in \Omega:\pm\alpha_{\min J}\geq0\}.
\]
 By doing this iteratively we obtain some $n_2\leq 2^{n}$ and $\Omega^{(i)}$, $i=1,\hdots,n_2$ such that
 \[
 	\Omega=\bigcup_{i=1}^{n_2} \Omega^{(i)}, \quad \bigcup_{i=1}^{n_2}J_{\Omega^{(i)}}=\emptyset.
 \]
 Hence, there is some $i\in\{1,\hdots,n_2\}$ such that $P\cap\Omega^{(i)}\neq\emptyset$. Since $J_{\Omega^{(i)}}=\emptyset$ there is some $l=(l_1,\hdots,l_n)$, $l_j\in\{-1,1\}$ such that $\alpha_j':=
 \inf_{\ve{\alpha}\in\Omega^{(i)}} l_j\alpha_j>-\infty$.

Define for $k\in \R$ the $n-1$-dimensional  hypersurface
\[
	Q_k:=\left\{\ve{\alpha}\in\R^n:\prod_{j\in \{1,\hdots,n\}}(1+l_j\alpha_{j}-\alpha_{j}')=k\right\}. 
\]
It follows that for  $k<1$ then $Q_k\cap \Omega^{(i)}=\emptyset$ and for all $k$, the set $Q_k\cap \Omega^{(i)}$ is closed and bounded. This together with that $\Omega^{(i)}\cap P\neq\emptyset$ means that there is a $k$ such that $Q_k\cap\Omega^{(i)}\cap P\neq\emptyset$ and $Q_{k'}\cap\Omega^{(i)}\cap P=\emptyset$ for all $k'<k$. Let $P'$ denote the  plane that tangents $Q_k$ in some point $\ve{\alpha}\in Q_k\cap\Omega^{(i)}\cap P$. From definition the $P'$ of $Q_k$ for each point in $\Omega^{(i)}$ is unique which implies that $\Omega^{(i)}\cap P\cap P'=\{\ve{\alpha}\}$. Since $P\cap P'\neq\emptyset$ it follows that $P\cap P'$ is a subspace of dimension at least $m-1$. Hence, $\Omega^{(i)}\cap P\cap P'=\{\ve{\alpha}\}$ yields that $\ve{\alpha}\in \Omega^{(i)}_{m-1}$. The result then follows from that $\ve{\alpha}\in \Omega^{(i)}_{m-1}\cap P\subset\Omega_{m-1}\cap P$ due to Definition \ref{mskel}. 
\end{proof}

\begin{prop}\label{detcmult}
Let $T$ be defined as the closure of \eqref{tmulthat}, with $n>2$ and let $\Omega$ be a closed set satisfying $W(A_1,\hdots,A_n)\subset \Omega\subset \overline{W(A_1)}\times\hdots\times \overline{W(A_n)}$. Then $\omega$ belongs to the set $W_\Omega(T)$ defined in \eqref{2enmu} if and only if $r:=\Ran F(\omega)=\Ran [F(\omega),G(\omega)]$ and  $\Omega_{n-r-1}\cap P\neq \emptyset$, where $P\subset \R^n$ is the $n-r$-dimensional solution subspace to \eqref{multeq}.
\end{prop} 
\begin{proof}
It follows that $\omega\in W_\Omega(T)$ if and only if \eqref{multeq} is solvable in $\Omega$ which means that  $\Ran F(\omega)=\Ran [F(\omega),G(\omega)]$, and  $ \Omega\cap P\neq\emptyset$. From Lemma \ref{lemfin1} and that the dimension of $P$ is $n-r$  the result follows.
\end{proof}

\begin{lem}\label{lemfin2} 
Let $\Omega\subset \R^n$ be a closed convex set and $P$ a $m$-dimensional subspace, where $0\leq m\leq n-1$. If $\partial\Omega\cap P\neq \Omega\setminus\partial\Omega\cap P= \emptyset$  then $\overline{\Omega_{m}}\cap P\neq\emptyset$.  \end{lem}
\begin{proof}

If $m=0$ the result follows from definition. Assume that $m\geq1$ and define $\Omega^{(i)}$ as in Lemma \ref{lemfin1}, which then is convex and $\partial\Omega^{(i)}\cap P\neq \Omega^{(i)}\setminus\partial\Omega^{(i)}\cap P= \emptyset$.

For simplicity assume that  $P= \{\ve{\alpha }\in\R^n:\alpha_{m+1}=\hdots =\alpha_{n} =0\}$, $\ve{\alpha}\in \Omega^{(i)}\Rightarrow \alpha_{m+1}\geq0$. Since, $\Omega^{(i)}$ is convex the problem can be written in this form in some coordinate system. Define the $m+1$-dimensional subspace $R:=\{\ve{\alpha}\in\R^n: \alpha_{m+2}=\hdots=\alpha_n=0\}$. For simplicity the notation $\ve{\alpha}=(\alpha_1,\hdots,\alpha_{m+1})$ is used since the omitted values are always $0$. Since $P\subset R$ it follows that $\Omega':=\Omega^{(i)}\cap R\neq\emptyset$.

 Let $\epsilon>0$ be some small number, then since no line is a subset of $\Omega'$ we can  assume that $(0,0,\hdots,0,\alpha_{m+1})\notin \Omega'$ for  $\alpha_{m+1}<\epsilon$. Additionally, by scaling (if necessary) we can assume  that $\ve{\alpha}'\in \Omega'\cap P$ with $\alpha_1'=\hdots=\alpha_m'=1$ and $\alpha_{m+1}'=0$. Define for $(x_1,\hdots,x_{m})\in [0,1]^m$ and $k>0$ the continuous function 
\[
	y_{k}(x_1,\hdots, x_m):=\max\left(\epsilon\left(1-2\sqrt{1-\sum_{j=1}^{m}\dfrac{(x_j-k)^2}{mk^2}}\right),0\right).
\]
Define the $m$-dimensional set
\[
	\Gamma_k:=\{(\alpha_1,\hdots,\alpha_m,y_k(\alpha_1,\hdots,\alpha_m)):(\alpha_1,\hdots,\alpha_{m})\in [0,1]^m\}.
\]

For $k$ close to zero it follows that  $\overline{(\Gamma_k\setminus P)}\cap\Omega'=\emptyset$ from closedness of $\Omega'$. Let $k>0$ be the smallest value such that there is some point $\ve{\alpha}\in \overline{(\Gamma_k\setminus P)}\cap\Omega'$. Such $k$ exists and is at most $4+\sqrt{12}$ since $\ve{\alpha}'\in\overline{(\Gamma_{4+\sqrt{12}}\setminus P)}\cap\Omega^{(i)}$.
Then there is a subspace $P'$ of dimension $m$ that tangents $\Gamma_k$ at $\ve{\alpha}$ that intersects only this point in $\Omega'$ since $\Omega' $ is convex and $\Gamma_k$ is strictly convex on $\overline{(\Gamma_k\setminus P)}$. Since $P'\subset R$ it follows that $P'$ intersects only one point in $\Omega^{(i)}_m$.

Hence, $\ve{\alpha}\in\Omega^{(i)}_{m}\subset\Omega_m$ and since ${\rm dist}(P,\ve{\alpha})\leq\epsilon$ where $\epsilon$ can be chosen arbitrary small we can choose $\ve{\alpha}$  arbitrary close to $P$. Finally, since each such $\ve{\alpha}$ is in a  bounded domain the limit of this construction as $\epsilon\rightarrow 0$ exists and is a point in $P$.
\end{proof}

Define for $T$ the partition of $\mathcal{C}$ as 
\begin{equation}\label{2partigen}
 \mathcal{C}^i=\{\omega\in \mathcal{C}: \Ran F(\omega)<2\},\quad
 \mathcal{C}^r= \mathcal{C} \setminus \mathcal{C}^i.
\end{equation}
This definition  generalizes  \eqref{2parti} to the case $n\geq2$. 

\begin{thm}\label{maintmult}
Let $T$ denote the closure of the operator function \eqref{tmulthat}, let sets of the type $W_X(T)$ be defined by \eqref{1encg} and let $C$ denote the set where $f^{(j)},g$ are continuous.
Let $\Omega$ be a closed set satisfying $W(A_1,\hdots,A_n)\subset \Omega\subset \overline{W(A_1)}\times\hdots\times \overline{W(A_n)}$ and $\Omega_m$ be defined by Definition \ref{mskel}. Then
\begin{itemize}
\item [{\rm (i)}] $W_{\Omega_{n-3}}(T)=W_\Omega(T)$,
\item[{\rm (ii)}] $W_{\Omega_{n-2}}(T)\cap \mathcal{C}^i=W_\Omega(T)\cap \mathcal{C}^i$,
\item[{\rm (iii)}] $W_{\overline{{\rm Co}(\Omega)_{n-2}}}(T)\cap \mathcal{C}^r\cap C\supset \partial W_{{\rm Co}(\Omega)}(T)\cap W_{{\rm Co}(\Omega)}(T)\cap \mathcal{C}^r\cap C$.
\end{itemize}
\end{thm}
\begin{proof}
{ \rm (i)}
From definition, $\Omega_{n-3}\subset \Omega$  and thus  $W_{\Omega_{n-3}}(T)\subset W_\Omega(T)$. Hence, the result follows if we can show that $W_{\Omega_{n-3}}(T)\supset W_{\Omega}(T)$. Assume $\omega\in W_{\Omega}(T)$, then from Proposition \ref{detcmult} it follows that \eqref{multeq} has a solution subspace $P\subset \R^n$ of dimension $m\geq n-2$ for $\omega$ and $P\cap \Omega\neq\emptyset$. From Lemma \ref{lemfin1} it then follows that $\Omega_{n-3}\cap P\supset \Omega_{m-1}\cap P\neq\emptyset$ and thus $\omega\in W_{\Omega_{n-3}}(T)$.

{ \rm (ii)}
The proof is analogous to the proof of {\rm (i)} with the exception that the solution subspace $P$ is of dimension at least $n-1$ which means that Lemma \ref{lemfin1} guarantees a solution in the smaller set $W_{\Omega_{n-2}}(T)$.

{ \rm (iii)} 
Assume that $\omega\in\partial W_{{\rm Co}(\Omega)}(T)\cap W_{{\rm Co}(\Omega)}(T)\cap\mathcal{C}^r\cap C$, from Proposition \ref{detcmult} there is an $n-2$-dimensional space $P\subset\R^n$  such that $t_{\ve{\alpha}}(\omega)=0$ for $\ve{\alpha}\in P$ and $P\cap{{\rm Co}(\Omega)}\neq\emptyset$. Since the solutions of a full rank problem is continuous in the coefficients and $f^{(j)}$ and $g$ are continuous at $\omega$, the property $\omega\in\partial W_{{\rm Co}(\Omega)}(T)$ implies that $\partial{{\rm Co}(\Omega)}\cap P\neq {{\rm Co}(\Omega)}\setminus\partial{{\rm Co}(\Omega)}\cap P= \emptyset$. Lemma \ref{lemfin2} then implies that $\overline{{\rm Co}(\Omega)_{n-2}} \cap P\neq\emptyset$, and thus $W_{\overline{{\rm Co}(\Omega)_{n-2}}}(T)$.

\end{proof}
Theorem \ref{maintmult} ({\rm iii}) can not be generalized to $\mathcal{C}^i$ without closing $W_{\overline{{\rm Co}(\Omega)_{n-2}}}(T)$ since, $W_\Omega(T)\cap \mathcal{C}^i$ is not closed in general, Remark \ref{imprem}.

\begin{rem}
If $\Omega:=W(A_1,\hdots,A_n)$ is closed, then Theorem \ref{maintmult} {\rm (i)}, and Proposition \ref{2Psats} yields that  $W_{\Omega_{n-3}}(T)=W_{{\rm Co}(W(A_1,\hdots,A_n))}(T)$.
\end{rem}

\begin{ex}
No general equality similar Theorem \ref{maint} {\rm (iii)} hold in the case when $n>2$. A very simple example that visualizes this is
\[
	T(\omega):=\omega^3+\omega^2A_1+\omega A_2+A_3,
\]
 where $W(A_1)\times W(A_2)\times W(A_3)\times [0,1]^3$. It then follows that $\mathcal{C}^i=\R$, $\mathcal{C}^r=\C\setminus \R$, $H=\C$ and that $\Omega_{1}$ consist of the edges of the cube $[0,1]^3$.
 Hence, $i/2\in W_{\Omega_{1}}(T) \cap \mathcal{C}^r$ since $i/2\in \mathcal{C}^r$ is a root of $\omega^3+\omega\frac{1}{4}$ and $(0,1/4,0)\in\Omega_1$. However $i/2\notin\partial W_\Omega(T)$ as can be seen in Figure \ref{4fig:manta}. 
 
 Conversely, there are non-trivial cases when $ W_{\overline{{\rm Co}(\Omega)_{n-2}}}(T)\cap \mathcal{C}^r\cap H = \partial W_\Omega(T)\cap \mathcal{C}^r\cap H$, with $n>2$.
\end{ex}

\section{Resolvent estimate}
In this section we present enclosures of the $\epsilon$-pseudonumerical range and upper estimates on the norm of the resolvent for the closure of \eqref{tmulthat}, which is a generalization of the results in \cite[Section 4]{ATEN}.  The $\epsilon$-pseudospectrum, $\sigma^\epsilon(T)$,  $\epsilon>0$ of an operator function $T$ is usually defined as $\omega\in \C$ such that $\omega\in \sigma(T)$ or $\|T(\omega)^{-1}\|>\epsilon^{-1}$. However, there is a number of equivalent definitions of $\sigma^\epsilon(T)$. For this section the most suiting definition is: 
\[
	\sigma^\epsilon(T):=\sigma(T)\cup\{\omega\in\doma\setminus\sigma(T): \exists u\in \dom(T(\omega))\setminus\{0\},\|T(\omega)u\|/\|u\|<\epsilon\}.
\]
 Pseudospectrum and resolvent estimates are useful in describing how well-behaved the operator  $T(\omega)^{-1}$ is for $\omega\in\doma\setminus\sigma(T)$, for further reading see \cite{MR2155029}. This is related to the $\epsilon$-numerical range given in Definition \ref{4pnr}, courtesy  to \cite[Definition 4.1]{ATEN}.
 \begin{defn}\label{4pnr}
For an operator function $T:\doma\rightarrow\Ls(\Hs)$ we define the $\epsilon$-pseudonumerical range as the set
\[
W^\epsilon(T):=W(T)\cup\{\omega\in\mathcal{C}\setminus W(T):\exists u\in\dom (T)\setminus\{0\}, |\ip{T(\omega)u}{u}|/\|u\|^2<\epsilon\}.
\]
\end{defn}

 If $T(\omega)=\widehat{T}-\omega$ then Definition \ref{4pnr} simplifies to 
\[
	W^\epsilon(T):=W(T)\cup\{\omega\in\doma\setminus W(T) :\exists \omega_0\in W(T): |\omega-\omega_0|<\epsilon\},
\]
and thus the definition of the $\epsilon$-pseudonumerical range is trivial for linear operator functions. 

The property, $\|T(\omega)u\|\|u\|\geq |\ip{T(\omega)u}{u}|$ for $u\in \dom (T)$ yields that
 $W^\epsilon(T)\setminus \sigma(T)\supset \sigma^\epsilon(T)\setminus \sigma(T)$. Hence, $\epsilon$-pseudonumerical range yields an upper estimate on the norm of the resolvent in $\doma\setminus W^\epsilon(T)$. 

However, just as in the case of numerical range, computing $W^\epsilon(T)$ explicitly for a general operator function is not possible.
For $T$ defined as the closure of \eqref{tmulthat}, $t_{\ve{\alpha}}$ defined  in \eqref{org2fun}, and a given set $X\subset \R^n$, let $W_X^\epsilon(T)\subset \C$ denote the set
\begin{equation}\label{1encge}
	W_X^\epsilon(T):=\{\omega\in \mathcal{C}: \exists \ve{\alpha}\in X, |t_{\ve{\alpha}}(\omega)|<\epsilon\}.
\end{equation}

\begin{lem}\label{5lemjoin}
Let $T$ be defined as the closure of \eqref{tmulthat}, $W(\ve{A})$ denote the joint numerical range of $A_1,\hdots,A_n$ and $W_{W(\ve{A})}^\epsilon(T)$ be a set on the form \eqref{1encge}. Then $W^\epsilon(T)=W_{W(\ve{A})}^\epsilon(T)$.
\end{lem}
\begin{proof}
From that $\dom(\ve{A})\subset \dom(T(\omega))$ it follows that $W^\epsilon(T)\supset W_{W(\ve{A})}^\epsilon(T)$, so only the converse has to be proven. Assume that $\omega\in W^\epsilon(T) $ then there is some unit vector $u\in\dom(T)$ such that $\ip{T(\omega)u}{u}=e$ where $|e|<\epsilon$. Since $\dom(\ve{A})$ is dense in $\Hs$ it follows that there is a sequence $\{v_i\}_{i=1}^\infty\in \dom(\ve{A})$ of unit vectors such that $v_i\rightarrow u$ and $\ip{T(\omega)v_i}{v_i}\rightarrow e$. Hence, by choosing $i$ large enough it follows that $|\ip{T(\omega)v_i}{v_i}|<\epsilon$. Define $\ve{\alpha}=(\ip{A_1v_i}{v_i},\hdots, \ip{A_nv_i}{v_i})\in W(\ve{A})$, then $|t_{\ve{\alpha}}(\omega)|<\epsilon$, which implies that $\omega\in W_{W(\ve{A})}^\epsilon(T)$. Therefore $W^\epsilon(T)\subset W_{W(\ve{A})}^\epsilon(T)$ and the result follows.
\end{proof}
 Lemma \ref{5lemjoin} states that  $W_{W(\ve{A})}^\epsilon(T)$ is the $\epsilon$-pseudonumerical range of $T$, \eqref{4pnr}, despite that $W(T)\neq W_{W(\ve{A})}(T)$ in general. Hence, for the operator function $T$ defined as the closure of \eqref{tmulthat} and $\Omega\supset W(A_1,\hdots,A_n)$ the set
\begin{equation}\label{5enrenc}
	W_\Omega^\epsilon(T)=W_\Omega(T)\cup\{\omega\in\doma\setminus W_\Omega(T) :\exists \ve{\alpha}\in\Omega, |t_{\ve{\alpha}}(\omega)|<\epsilon\},
\end{equation}
gives an enclosure of $W^\epsilon(T)$, where $t_{\ve{\alpha}}$ is given by \eqref{org2fun}. 

\begin{center}
\begin{figure}
\includegraphics[width=12.5cm]{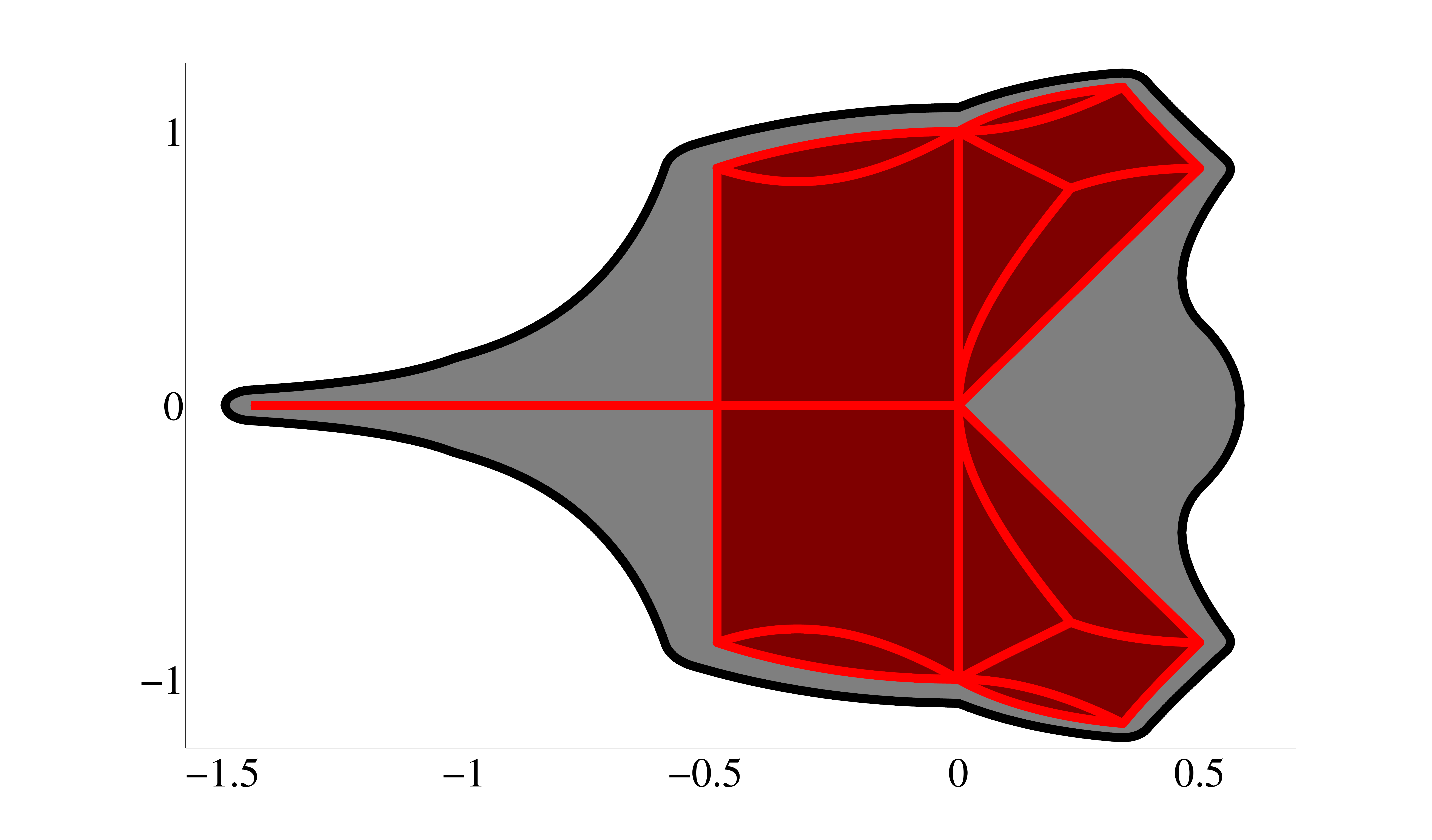}
\caption{The inner part denotes $W_\Omega(T)$ where the bright lines denote $W_{\Omega_{1}}(T)$. The outer area denote $W_\Omega^{\frac{1}{5}}(T)$.
Here, $T(\omega):=\omega^3+\omega^2A_1+\omega A_2+A_3$, with $\Omega=W(A_1)\times W(A_2)\times W(A_3)=[0,1]^3$. }\label{4fig:manta}
\end{figure}
\end{center}

If $\Omega$ is a bounded set and $f^{(j)},g$, $j=1,\hdots,n$ are continuous it follows that 
\begin{equation}\label{5niceeps}
	 \partial W_\Omega^\epsilon(T) \subset\{\omega\in\doma:\inf_{\ve{\alpha}\in \Omega } |t_{\ve{\alpha}}(\omega)|=\epsilon \},
\end{equation}
with equality if they are holomorphic and linearly independent. However, if $\Omega$ is unbounded then the sequence $\{\ve{\alpha}^{(i)}\}_{i=1}^\infty\in\Omega$ such that $ |t_{\ve{\alpha}^{(i)}}(\omega)|\rightarrow \epsilon$ is not always convergent. The results in Theorem \ref{3emaint} {\rm (iii), (iv)} are generalizations of \eqref{5niceeps} to unbounded $\Omega$.\begin{thm}\label{3emaint}
Let $T$ denote the closure of \eqref{tmulthat} and  sets of the type $W^\epsilon_X(T)$ be defined by \eqref{1encge}, let $\Omega$ be a closed set satisfying $W(A_1,\hdots,A_n)\subset \Omega\subset \overline{W(A_1)}\times\hdots\times \overline{W(A_n)}$ and let $\Omega_m$ be defined by Definition \ref{mskel}.  Furthermore, denote by $C$ the set where $f^{(j)},g$, $j=1,\hdots,n$ are continuous, and $H$ as the set where they are holomorphic and linearly independent. If either $n=2$ or $\Omega$ is convex or bounded define  $m:=n-2$, otherwise define  $m:=n-3$. Then \begin{itemize}
 \item[{\rm (i)}] $W_{\Omega_{m}}^\epsilon(T)\setminus W_\Omega(T)=W_\Omega^\epsilon(T)\setminus W_\Omega(T)$,
 \item[{\rm (ii)}] $W_{\Omega_{n-2}}^\epsilon(T)\cap \doma^i\setminus W_\Omega(T)=W_\Omega^\epsilon(T)\cap \doma^i\setminus W_\Omega(T)$,
  \item[{\rm (iii)}] $ \{\omega\in\partial W_\Omega^\epsilon(T) \cap C:\epsilon'>\epsilon,\omega\notin\partial W_\Omega^{\epsilon'}(T) \}\subset\{\omega\in C: \inf_{\ve{\alpha}\in \Omega } |t_{\ve{\alpha}}(\omega)|=\epsilon \}$,
    \item[{\rm (iv)}] $\{\omega\in\partial W_\Omega^\epsilon(T) \cap H:\epsilon'>\epsilon ,\omega\notin\partial W_\Omega^{\epsilon'}(T) \}=\{\omega\in H: \inf_{\ve{\alpha}\in \Omega } |t_{\ve{\alpha}}(\omega)|=\epsilon \}$.
 \end{itemize}
\end{thm}
\begin{proof}
{\rm (i)-(ii)} From definition it follows that $W_{\Omega_{m}}^\epsilon(T)\setminus W_\Omega(T)\subset W_\Omega^\epsilon(T)\setminus W_\Omega(T)$. 
Assume that $\omega\in W_\Omega^\epsilon(T)\setminus W_\Omega(T)$, then for some $\ve{\alpha}\in \Omega$,  $t_{\ve{\alpha}}(\omega)=e$ where $0<|e|<\epsilon$. Then similar to Proposition \ref{detcmult} there is a space $P\subset\R^n$ such that $t_{\ve{\alpha}}(\omega)=e$ for $\ve{\alpha}\in P$, where the dimension of $P$ is $n-2$ if $\omega\in\doma^r$ and at least $n-1$ if $\omega\in\doma^i$. From Lemma \ref{lemfin1} it follows that $\omega\in W_{\Omega_{n-3}}^\epsilon(T)\setminus W_\Omega(T)$ if $\omega\in\doma^r$ and $\omega\in W_{\Omega_{n-2}}^\epsilon(T)\setminus W_\Omega(T)$ if $\omega\in\doma^i$ similar to Theorem \ref{maintmult} {\rm (i)-(ii)}. This proves {\rm (ii)} and also and implies {\rm (i)} when $\omega\in\doma^i$ or $m=n-3$. 

Hence, we only have to show the result for $\omega\in \mathcal{C}^r$ and $m=n-2$.  It follows similar to Proposition \ref{detcmult} that for all $r\in [0,1]$ there is a $n-2$-dimensional solution space $P_{r}\subset\R^n$ such that $t_{\ve{\alpha}}(\omega)=re$ for $\ve{\alpha}\in P_r$. From construction these spaces are parallel.  There are three cases to investigate: $n=2$, $\Omega$ is convex or $\Omega$ is bounded. Define $\Omega^{(i)}$ as in the proof of Lemma \ref{lemfin1} such that $P_1\cap \Omega^{(i)}\neq\emptyset$. Since $\omega\notin W_\Omega(T)$ it follows that $P_0\cap \Omega^{(i)}=\emptyset$. 

Assume that $n=2$. Then it follows that $P_r$ is only a point for all $r\in[0,1]$. Hence there is a $r'\in(0,1)$  such that $P_{r'}=\{\ve{\alpha}'\}\subset \partial \Omega \subset \Omega_0$ and the result follows in this case since $|t_{\ve{\alpha}'}(\omega)|=|r'e|<\epsilon$. Hence, in the following assume that $\Omega$ is either bounded or convex.

Similar to the proof of Lemma \ref{lemfin2}, we assume that $P_r= \{\ve{\alpha }\in\R^n:\alpha_{n-1}=r,\alpha_{n} =0\}$, which will be true in some coordinate system. Define the $n-1$-dimensional subspace $R:= \{\ve{\alpha }\in\R^n:\alpha_{n} =0\}$, and use the notation $\ve{\alpha}=(\alpha_1,\hdots,\alpha_{n-1})$ for $\ve{\alpha}\in R$. Define $\Omega':=\Omega^{(i)}\cap R$, which then is non-empty.

Assume that  $\Omega$ is bounded, then  $\Omega'$ is bounded as well. Define for $\epsilon>0$, $k\geq0$ and for real $x_j$, $j=1,\hdots,n-2$ the function
\[
	y_k(x_1,\hdots,x_{n-2})=k+\epsilon\sum_{j=1}^{n-2}x_j^2,
\]
and the $n-2$-dimensional surface in $R$
\[
	\Gamma_k:=\{(\alpha_1,\hdots,\alpha_{n-2},y_k(\alpha_1,\hdots,\alpha_{n-2})):(\alpha_1,\hdots,\alpha_{n-2})\in \R^{n-2}\}.
\]
For each $\epsilon>0$, the surface $\Gamma_k$ has a unique tangent. Furthermore, if $\epsilon$ is small enough then $\Gamma_0\cap \Omega'=\emptyset$ since $\Omega'$ is bounded. For each small $\epsilon>0$ there is a unique $k\leq 1$ such that there is an  $\ve{\alpha}'\in\Gamma_k\cap \Omega'$ and $\Gamma_{k'}\cap \Omega'=\emptyset$ for $k'<k$. It also follows that if  $\epsilon>0$ is small enough then $|t_{\ve{\alpha}'}(\omega)|<\epsilon$ and the tangent plane of $\Gamma_k$ at $\ve{\alpha}'$ does not intersect any other point in $\Omega'$, since $\Gamma_k$ is close to the plane $P_k$. The result then follows from that $\Omega'\subset\Omega^{(i)}$.

 Assume that $\Omega$ is convex, then  $\Omega'$ is convex as well. If there is some smallest $r'\in(0,1)$ such that $P_{r'}\cap \Omega'\neq\emptyset$ then Lemma \ref{lemfin2} implies that there is some $\ve{\alpha}'\in\Omega_{n-2}$ such that $|t_{\ve{\alpha}'}(\omega)|<\epsilon$ due to $t$ being continuous  in $\ve{\alpha}$. Hence, we only have to show the result when no such $r'$ exists. In that case there is a largest $r'$  such that $P_{r'}\cap \Omega'=\emptyset$. The distance between $P_{r'}$ and  $\Omega'$ must then be $0$ due to continuity. Choose a point in $\ve{\alpha}'\in\partial\Omega'$ such that ${\rm dist}(P_{r'},\ve{\alpha}')<\epsilon'$ for some $\epsilon'>0$. Let $P'$ denote a $n-2$-dimensional hyper plane  that tangents $\Omega'$ at $\ve{\alpha}'$, (if $\Omega'$ is smooth at that point this tangent is unique).  Hence, Lemma \ref{lemfin2} implies that $ \overline{\Omega'_{n-2}}\cap P'\neq \emptyset$. From the construction at least one point in  $\ve{\alpha}\in\overline{\Omega'_{n-2}}\cap P'$, must satisfy ${\rm dist}(P_{r'},\ve{\alpha})\leq \epsilon'$. Hence, if $\epsilon'<r'-1$ then $|t_{\ve{\alpha}}(\omega)|<\epsilon$, and the result follows from that $\Omega'\subset\Omega^{(i)}$.

 {\rm (iii)} 
 Assume that $\omega\in \{\omega\in\partial W_\Omega^\epsilon(T) \cap C:\epsilon'>\epsilon ,\omega\notin\partial W_\Omega^{\epsilon'}(T) \}$. Assume that $\inf_{\ve{\alpha}\in\Omega}|t_{\ve{\alpha}}(\omega)|<\epsilon$. Then there is some $\ve{\alpha}'\in\Omega$ such that $|t_{\ve{\alpha}'}(\omega)|<\epsilon$. By continuity there is a neighborhood $N$ of $\omega$ such that $|t_{\ve{\alpha}'}(\omega')|<\epsilon$ for $\omega'\in N$. This contradicts that $\omega\in\partial W_\Omega^\epsilon(T)$. Hence,  $\inf_{\ve{\alpha}\in\Omega}|t_{\ve{\alpha}}(\omega)|\geq\epsilon$. If $\inf_{\ve{\alpha}\in\Omega}|t_{\ve{\alpha}}(\omega)|=\epsilon'>\epsilon$ then since $ \omega\notin W_\Omega^{\epsilon'}(T)\supset W_\Omega^\epsilon(T)$ it follows that $\omega=\partial W_\Omega^{\epsilon'}(T)$ which is a contradiction and {\rm (iii)} follows. 
 
 {\rm (iv)} 
 Assume that $\omega\in H$. Then $\{\omega\in\partial W_\Omega^\epsilon(T) \cap H:\epsilon'>\epsilon ,\omega\notin\partial W_\Omega^{\epsilon'}(T) \}\subset\{\omega\in H: \inf_{\ve{\alpha}\in \Omega } |t_{\ve{\alpha}}(\omega)|=\epsilon \}$
  follows directly from $H\subset C$ and {\rm (iii)}.
  Assume that $\inf_{\ve{\alpha}\in\Omega}|t_{\ve{\alpha}}(\omega)|=\epsilon$ and let $\{\ve{\alpha}^{(i)}\}_{i=1}^\infty$ be a sequence in $\Omega$ such that $|t_{\ve{\alpha}^{(i)}}(\omega)|\rightarrow\epsilon$.  Since $\epsilon>0$ and  $t_{\ve{\alpha}^{(i)}}$ is non-constant and holomorphic around $\omega$, the minimum modulus principle states that for each neighborhood $N$ of $\omega$, there is a $\omega^{(i)}\in N$ such that $\epsilon^{(i)}:=|t_{\ve{\alpha}^{(i)}}(\omega^{(i)})|<|t_{\ve{\alpha}^{(i)}}(\omega)|$. 
   Additionally since $f^{(j)},g$, $j=1,\hdots,n$ are linearly independent $t_{\ve{\alpha}^{(i)}}$ does not approach a function locally constant $\epsilon$ around $\omega$. This implies that $\omega^{(i)}\in N$ can be chosen such that $\epsilon^{(i)}<\epsilon$ for $i$ large enough and thus $\omega^{(i)}\in W_\Omega^\epsilon(T)$. $N$ can be chosen arbitrary small which implies that $\omega\in \partial W_\Omega^\epsilon(T)$. Additionally for each $\epsilon'>\epsilon$, there is neighborhood $N'$ of $\omega$ and an $i$ such that  $|t_{\ve{\alpha}^{(i)}}(\omega')|<\epsilon'$ for $\omega'\in N'$. Hence, $\omega\notin \partial W_\Omega^{\epsilon'}(T)$. This proves $\{\omega\in\partial W_\Omega^\epsilon(T) \cap H:\epsilon'>\epsilon ,\omega\notin\partial W_\Omega^{\epsilon'}(T) \}\supset\{\omega\in H: \inf_{\ve{\alpha}\in \Omega } |t_{\ve{\alpha}}(\omega)|=\epsilon \}$, and ${\rm (iv)}$ follows.
\end{proof}

\begin{rem}
 In Theorem \ref{3emaint} {\rm (iii), (iv)} it is equivalent to verify that $\omega\notin\partial W_\Omega^{\epsilon'}(T) $ as $\epsilon'\rightarrow \epsilon$ from above due to monotonicity. Hence, these sets can be defined as limits. 
\end{rem}

\begin{rem}
From Theorem \ref{maintmult} {\rm (i)} and Theorem \ref{3emaint} {\rm (i)} it follows that $ W_{\Omega_{n-3}}^\epsilon(T)= W_\Omega^\epsilon(T)$. 
\end{rem}

\begin{center}
\begin{figure}
\includegraphics[width=12.5cm]{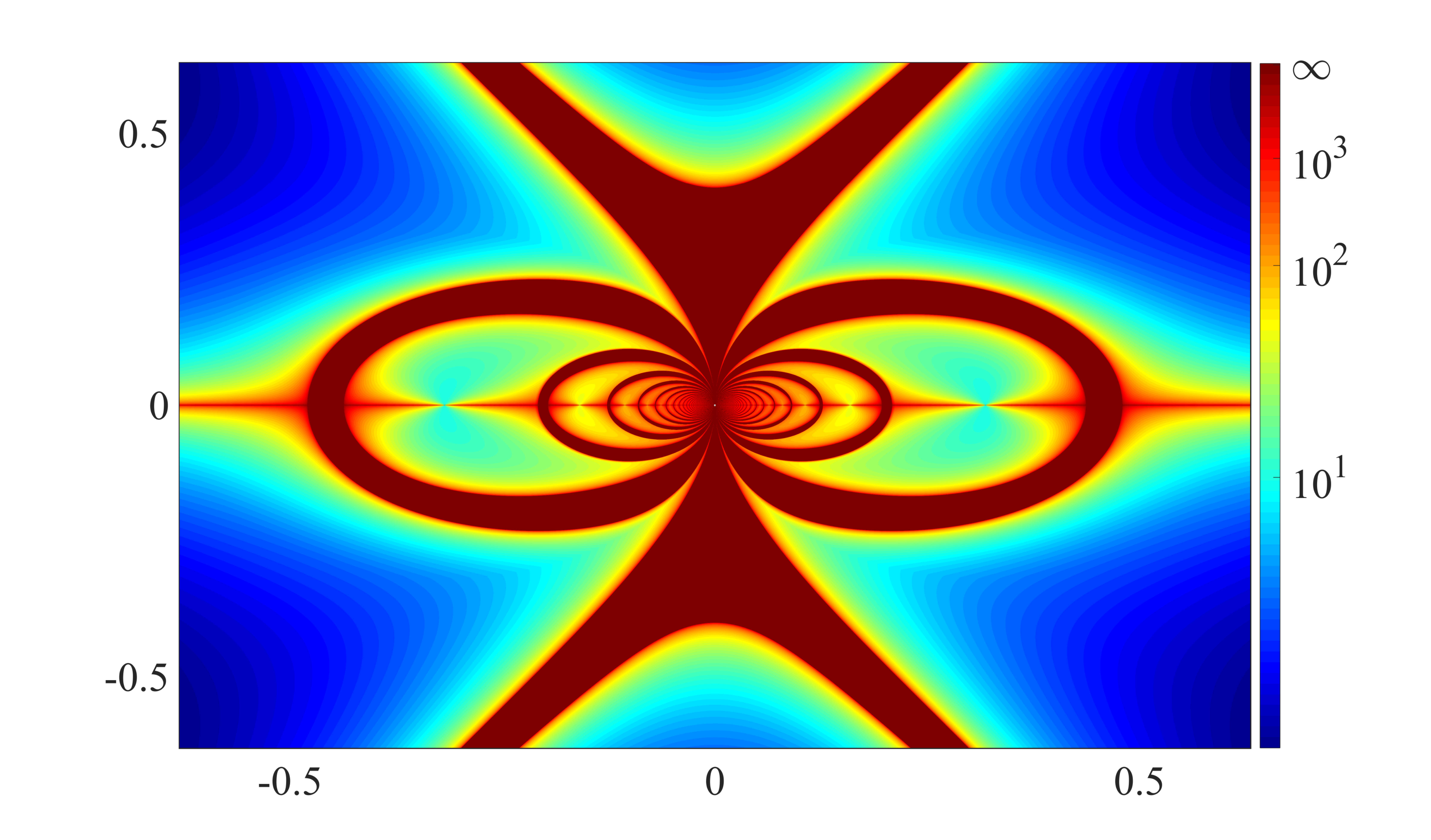}
\caption{Define the operator function $T(\omega):=A_1\sin\frac{1}{\omega}+A_2+\omega^2$ on $\mathcal{C}:=\C\setminus\{0\}$ with $\overline{W(A_1)}=(-\infty,\infty)$ and  $\overline{W(A_2)}=[0,4/25]$. In the Figure, $W_{\Omega}(T)$, $\Omega:=(-\infty,\infty)\times[0,4/25]$ is visualized in dark red and the rest shows an estimate of $\|T(\omega)\|^{-1}$. 
}\label{2fig:sinclust}
\end{figure}
\end{center}

Consider $F(\omega)$ and $G(\omega)$ defined in \eqref{multeq} and define $\epsilon_\omega$ as the least squares solution:
\begin{equation}\label{3tbound}
	\epsilon_\omega:=\inf_{\ve{\alpha}\in\Omega}|t_{\ve{\alpha}}(\omega)|=\inf_{\ve{\alpha}\in\Omega}\|F(\omega)\ve{\alpha}-G(\omega)\|.
\end{equation}
By solving the finite dimensional least squares problem \eqref{3tbound} an estimate of the resolvent of $T(\omega)$ is given in  Proposition \ref{5lastprop}.
\begin{prop}\label{5lastprop}
Let $T$ be defined as the closure of \eqref{tmulthat} and $W_\Omega^\epsilon(T)$  be defined by \eqref{5enrenc}, where $\Omega$ is a closed set satisfying $W(A_1,\hdots,A_n)\subset \Omega\subset \overline{W(A_1)}\times\hdots\times \overline{W(A_n)}$.
Denote by $\epsilon_\omega$ the least squares solution \eqref{3tbound}. Then $\omega\in W_\Omega^\epsilon(T)$ for $\epsilon>0$ if and only if $\epsilon_\omega< \epsilon$. Further, for $\omega\in\doma\setminus W_\Omega(T)$ the resolvent estimate $\|T(\omega)^{-1}\|\leq \epsilon_\omega^{-1}$ holds.
\end{prop}
\begin{proof}
From the definition of $\epsilon$-pseudonumerical range, \eqref{5enrenc}, it follows directly that $\omega\in W_\Omega^\epsilon(T)$ if and only if $\epsilon> \epsilon_\omega$. Further, for all $u\in \dom(T(\omega))$, $\|T(\omega)u\|/\|u\|\geq|\ip{T(\omega)u}{u}|/\|u\|^2\geq \epsilon_\omega$, which yields that $\|T(\omega)^{-1}\|\leq \epsilon_\omega^{-1}$.
\end{proof}

\vspace{3.5mm}

{\small
{\bf Acknowledgements.} \ 
The author gratefully acknowledge the support of the Swedish Research Council under Grant No.\ $621$-$2012$-$3863$.  The author also would like to thank Christian Engstr\"om for important feedback during the preparation of this article.
}

\bibliographystyle{alpha}
\bibliography{Bibliography}

\end{document}